\documentclass{amsart}
\usepackage[utf8]{inputenc}
\usepackage{latexsym,amsxtra,amscd,ifthen}
\usepackage{amsfonts}
\usepackage{color,graphicx}
\usepackage{verbatim}
\usepackage{amsmath}
\usepackage{amsthm}
\usepackage{mathrsfs}
\usepackage{amssymb}
\usepackage{url}
\usepackage{ulem}
\usepackage{hyperref}
\usepackage{float}

\usepackage{tikz}
\usetikzlibrary{arrows,shapes,chains}
\usepackage{tikz-cd}

\usepackage{yhmath}

\theoremstyle{plain}

\newtheorem{theorem}{Theorem}
\newtheorem{lemma}[theorem]{Lemma}
\newtheorem{proposition}[theorem]{Proposition}
\newtheorem{corollary}[theorem]{Corollary}

\numberwithin{theorem}{section}
\numberwithin{equation}{theorem}

\theoremstyle{definition}
\newtheorem{definition}[theorem]{Definition}
\newtheorem{example}[theorem]{Example}
\newtheorem{remark}[theorem]{Remark}

\newtheorem*{question*}{Question}

\newcommand{\Z}{\mathbb{Z}}

\newcommand{\kk}{\Bbbk}

\DeclareMathOperator{\Mod}{Mod}

\DeclareMathOperator{\Hom}{Hom}

\DeclareMathOperator{\sMod}{sMod}

\DeclareMathOperator{\sotimes}{\overset{s}{\otimes}}
\DeclareMathOperator{\shom}{\overset{s}{\Hom}}
\DeclareMathOperator{\Ring}{\mathtt{Ring}}
\DeclareMathOperator{\Ab}{\mathtt{Ab}}
\DeclareMathOperator{\sAb}{\mathtt{sAb}}
\DeclareMathOperator{\sRing}{\mathtt{sRing}}

\DeclareMathOperator{\kAlg}{\mathtt{\Bbbk-Alg}}

\begin{document}

\title{Introduction to Separable Simplicial Algebras}

\author{Xuan Huang, Garcia Sun, Eric Zhang, and James J. Zhang}

\address{Xuan Huang: Department of Mathematics, Box 354350,
University of Washington, Seattle, Washington 98195, USA}

\email{xhuang9@uw.edu}

\address{Garcia Sun: Department of Mathematics, Box 354350,
University of Washington, Seattle, Washington 98195, USA}

\email{zhiqisun@uw.edu}

\address{Eric Zhang: Department of Mathematics, Box 354350,
University of Washington, Seattle, Washington 98195, USA}

\email{zhangj20@uw.edu}

\address{James J. Zhang: Department of Mathematics, Box 354350,
University of Washington, Seattle, Washington 98195, USA}

\email{zhang@math.washington.edu}

\begin{abstract}
We study separable simplicial algebras and Azumaya simplicial algebras.
\end{abstract}









\maketitle


\section*{Introduction}
\label{yysec0}

\subsection{Motivation}
\label{yysec0.1}
Let $\Bbbk$ be a base commutative ring. If $\Bbbk$ is a field,
separable algebras over $\Bbbk$ are exactly 0-dimensional twisted 
Calabi-Yau algebras
\cite[Theorem 1.6]{RR22}; as such, they possess good homological properties 
(with minimal dimension). While the structure of separable 
algebras over a field is well-understood, separable 
algebras over a general commutative base ring are also 
extensively studied due to their foundational role in several areas such as algebra, 
number theory, and algebraic geometry. A special class of separable algebras are Azumaya 
algebras, which serve as key ingredients in defining the Brauer group of a commutative ring.

In this paper, we study separable simplicial algebras and Azumaya simplicial algebras and 
introduce the Brauer group of a simplicial commutative ring at the end. One of our ultimate
goals is to find a useful 
simplicial version of the Brauer group \cite{HSZZ27}.

\subsection{Review of previous
results}
\label{xxsec0.2}
The classical theory of separable algebras can be found in \cite{AG1960, DI1971, Ford2017}.
Pareigis initiates the study of separable algebras in symmetric monoidal categories \cite{PareigisIV},
which is extended to the setting of a braided monoidal category by Van Oystaeyen and Y. Zhang \cite{VOZ1998}.
Later Borceux and Vitale generalize the notion to separable enriched categories \cite{BV2002}. 
Balmer considers separability in triangulated categories \cite{Balmer2011}, 
and recently Naumann, Pol, and Ramzi push the notion of separable commutative algebras to symmetric monoidal stable $\infty$-categories \cite{NPR2026}.
We also note there are closely related studies of Azumaya algebras and Brauer groups in the monoidal categorical setting \cite{AMS2007,  BBW2009, BT2015, FisherPalmquist1975, Vitale1996, Toen2012}
along with the notion of separable functors \cite{CMZ2002, NdBvO1989}.

Although the separability of an algebra is almost unanimously defined by the splitting of the multiplication map (as a bimodule map), 
an Azumaya algebra can be defined in two natural ways: 
either as a central separable algebra or as a certain ``generator'' that is Morita equivalent to the base ring.
These two notions are equivalent in the classical setting by the celebrated Auslander-Goldman theorem (\cite[Theorem 2.1]{AG1960} and \cite[Theorem 1.2]{BV2002}).
In the more general monoidal categorical setting,
the Morita-theoretical definition of an Azumaya algebra (for instance, see \cite[p.~63]{FisherPalmquist1975}, \cite[p.~114]{PareigisIV}, and \cite[Proposition 3.1]{VOZ1998}) is the most-often studied.

\subsection{Why the 
simplicial setting?}
\label{yysec0.3}

Commutative simplicial algebras/rings 
have been used in algebraic topology, 
homological algebra, and derived 
algebraic geometry. We choose to 
further investigate separable algebras 
in the simplicial setting because we 
are searching for useful and 
computable invariants of simplicial 
algebras. Our hope is to 
understand the correct 
definition of the Brauer group in 
the simplicial setting,
which is also the main goal of
our forthcoming paper \cite{HSZZ27}.

Passing from ordinary algebras to simplicial algebras is often a first step toward a homotopical generalization of classical theory
\cite{SS2000, SS2003,Toen2012, TV2008};
for instance, To\"en defines  derived Azumaya algebras over a commutative simplicial ring \cite{Toen2012}.
In general, simplicial algebras carry homotopical information through their homotopy groups while retaining the familiar algebraic language of rings, modules, and homological algebra.
In this paper we generalize 
several classical results concerning separable algebras to the simplicial setting. This becomes possible because
we have more tools in the simplicial setting than in the general setting
of monoidal categories.

For comparison, let us consider again
the notion of an Azumaya algebra
in a general setting
of monoidal categories as well as 
in the 
setting of this paper.
Let $R$ be a commutative simplicial algebra. Our Azumaya algebra over $R$
is defined to be a central
separable $R$-algebra (Definition 
\ref{yydef6.1}), which is also 
Morita-theoretical by Theorem \ref{yythm0.2}.
In the general setting, 
one can fully 
characterize the Morita-theoretical Azumaya algebras as in \cite[Theorem 4]{FisherPalmquist1975},
it remains unknown if such a notion agrees with the central separable version.
This deficiency leads Pareigis to define the $1$-Azumaya algebra (\cite[p.~115]{PareigisIV}, Morita-theoretically) and the $2$-Azumaya algebra (\cite[p.~121]{PareigisIV}, central-separability and some extra requirement).
Pareigis notes in \cite[p.~121]{PareigisIV} that it is generally unclear if $1$-Azumaya algebras and $2$-Azumaya algebras coincide.
If the unit is projective, then they agree \cite[p.~124]{PareigisIV}. 
But $2$-Azumaya algebras are still not the central-separable ones.

Following Pareigis' work,
Van Oystaeyen and Y. Zhang define two versions of Brauer groups using central separable algebras (\cite[p.~123]{VOZ1998}) and Azumaya algebras (\cite[p.~108]{VOZ1998}).
Their unifying theorem \cite[Theorem 4.9]{VOZ1998} equates central-separable algebras with Azumaya algebras but relies on the assumption that the base symmetric monoidal category $\mathcal{C}$ is $\Bbbk$-linear and the commutative ring $\Bbbk$ is the unit of $\mathcal{C}$. 
Borceux and Vitale's result \cite[Theorem 3.4]{BV2002} is another unifying theorem equating central-separable ``strong generators'' with Azumaya algebras,
which is more akin to \cite[Theorem 12]{PareigisIV} in spirit.
When $R$ is a commutative simplicial ring, the category of simplicial
modules over $R$ is a closed symmetric monoidal category. But the category 
of simplicial modules over $R$ does not satisfy the crucial assumption of Van Oystaeyen and Y. Zhang (\cite[p.~112]{VOZ1998}) when $R$ is not a constant simplicial ring.
So the central-separable/Azumaya equivalence (Theorem \ref{yythm0.2}) 
does not follow from the previous 
work such as \cite{VOZ1998}.

\subsection{Results of this paper}
\label{yysec0.4}

Let $R$ be a commutative simplicial ring and $R\sMod$ be the category of simplicial modules over $R$.
The tensor product in $R\sMod$ is denoted by $\sotimes_R$ and the unit is $R$.
Let $A$ be a simplicial algebra over $R$ (=a monoid in $R\sMod $) and $A^{o}$ be the opposite algebra (see Section \ref{yysec1}).
Let $A^e = A \sotimes_R A^{o}$.
Then $A$ is {\it separable} 
over $R$ if the multiplication map $\mu: A^e \to A $ splits.

\subsubsection{Basics}
We first formulate the simplicial centralizers (Definition \ref{yydef3.7}) and enriched separable functors (Definition \ref{yydef3.13}) to characterize simplicial separability.

\begin{proposition}
[Propositions \ref{yypro3.10} and \ref{yypro3.16}]
\label{yypro0.1}
    Let $A$ be a  simplicial algebra over $R$.
    The following are equivalent.
    \begin{enumerate}
        \item[(1)] $A$ is separable over $R$.

        \item[(2)] The centralizer functor $(-)^A: A^e\sMod \to R\sMod$ is exact.

        \item[(3)] The forgetful functor $\mathcal{U}: A\sMod \to R\sMod$ is a separable $R\sMod$-enriched functor.
    \end{enumerate}
\end{proposition}

\subsubsection{Auslander-Goldman Theorem}
Following the classical theory \cite{Ford2017}, 
a simplicial algebra $A$ over $R$ 
is {\it Azumaya} if $A$ is central separable over $R$.
Based on the enriched Morita theory \cite{VOZ1998, PareigisIII},
we characterize the simplicial $R$-progenerator (Definition \ref{yydef5.1}) as a simplicial $R$-module $M$ such that $M$ is a direct summand of some $R^{\oplus t}$ and $R$ is a direct summand of some $M^{\oplus t'}$ as simplicial $R$-modules (Lemma \ref{yylem5.3}).
We show a simplicial version of Auslander-Goldman Theorem:
central-separable Azumaya algebras (condition (1) of Theorem \ref{yythm0.2}) agree with Morita-theoretical Azumaya algebras (condition (3) of Theorem \ref{yythm0.2}) for simplicial algebras.
This fully resolves the question left open in \cite{  BV2002,VOZ1998,PareigisIV} for simplicial algebras.

\begin{theorem}
[Theorem \ref{yythm6.4}]
\label{yythm0.2}
{\rm{(Unifying Theorem)}}
Let $A$ be a simplicial algebra over $R$.
The following are equivalent.
\begin{enumerate}
\item[(1)] 
$A$ is Azumaya (central separable) over $R$.
\item[(2)] 
$A$ is a simplicial $A^e$-progenerator and $A$ is central over $R$.
\item[(3)] 
$A$ is a simplicial $R$-progenerator and $A^e \to \shom_R(A,A)$ given by left $A^e$-actions is an isomorphism of simplicial $R$-algebras.
\end{enumerate}
\end{theorem}

We are wondering if the above unifying
theorem holds for other non-simplicial
setting.
Along this line, we also have

\begin{theorem}
[Theorem \ref{yythm6.6}]
\label{yythm0.3}
    Let $A$ be a simplicial algebra over $R.$
    Then 
    $A$ is separable if and only if $A$ is Azumaya over the simplicial center $Z(A)$ {\rm{(}}Definition \ref{yydef1.5}{\rm{)}} and $Z(A)$ is separable over $R$.
\end{theorem}

\subsubsection{Classifications}
We classify certain families of separable simplicial algebras which indicate some rigidity of separability for simplicial algebras. In this
subsection, $\Bbbk$ is a base commutative ring.

\begin{theorem}[Theorem \ref{yythm2.11}]
\label{yythm0.4}
Suppose $R=\Bbbk$ is a commutative ring,
viewed as a constant simplicial ring.
Let $A$ be a  simplicial algebra over $R$.
Then $A$ is separable over $\Bbbk$ if and only if $A\cong A_0\Delta^0 :=\Bbbk\Delta^0 \sotimes_{\Bbbk} A_0$ and $A_0$ is an ordinary separable 
algebra over $\Bbbk$.
\end{theorem}

Suppose $\Bbbk$ is a field.
While an ordinary separable 
$\Bbbk$-algebra has global dimension zero, the category of left simplicial modules over a separable simplicial 
$\Bbbk$-algebra has infinite global dimension
(Corollary \ref{yycor2.12}).

\begin{proposition}[Proposition \ref{yypro6.7}]
\label{yypro0.5}
    Let $A$ be a simplicial algebra over $R$.
    Then $A$ is Azumaya over $R$ if and only if $A \cong R \sotimes_{R_0} A_0$ as simplicial $R$-algebras and $A_0$ is an Azumaya $R_0$-algebra.
\end{proposition}

\begin{proposition}[Corollary \ref{yycor6.8}]
\label{yypro0.6}
        Let $A$ be a  simplicial algebra over $R$.
        The following are equivalent.
        \begin{enumerate}
            \item[(1)] $A$ is separable over $R$.

            \item[(2)]  $Z(A)$ is separable over $R$, 
            $A_0$ is an Azumaya $Z(A)_0$-algebra,
            and
            $$A\cong Z(A) \sotimes_{Z(A)_0} A_0$$ as simplicial $Z(A)$-algebras.
        \end{enumerate} 
\end{proposition}

    Proposition 0.6 shows that a separable simplicial \(R\)-algebra is determined by the following data:
\begin{enumerate}
    \item  a commutative separable simplicial $R$-algebra (the center $Z(A)$).

    \item an ordinary Azumaya $Z(A)_0$-algebra (the algebra $A_0$).
\end{enumerate} 
The second datum is controlled, up to Morita equivalence, by the Brauer group $Br(Z(A)_0)$.
Thus the remaining problem is to classify the commutative separable simplicial $R$-algebras that can occur as centers.

\begin{lemma}[Lemma \ref{yylem6.10}]
\label{yylem0.7}
Let $C$ be a commutative simplicial $R$-algebra.
Then $C\cong Z(A)$ for some separable simplicial $R$-algebra $A$ if and only if $C$ is separable over $R.$
\end{lemma}

\subsubsection{Commutative simplicial 
separable algebras}
We characterize commutative simplicial 
separable algebras as follows.

\begin{theorem}[Theorem \ref{yythm6.11}]
\label{yythm0.8}
Let $C$ be a commutative simplicial $R$-algebra.
The following are equivalent.
\begin{enumerate}
\item[(1)] $C$ is separable over $R$. 
\item[(2)] $C_0$ is separable over $R_0$ and the canonical simplicial ring morphism
$$\phi: R \sotimes_{R_0} C_0 \to C \qquad  r\otimes c \mapsto rs_0^n(c),
\qquad \forall \;r\otimes c \in (R \sotimes_{R_0} C_0)_n
$$ 
is a ring epimorphism levelwise.
\end{enumerate}
\end{theorem}


\subsubsection{Brauer group}
The definition of the Brauer group of a commutative simplicial ring $R$ is
given in Definition \ref{yydef7.2}. 
We show that it is isomorphic to the Brauer group of the degree-zero component $R_0.$

\begin{theorem}
[Theorem \ref{yythm7.4}]
\label{yylem0.9}
{\rm{(Triviality Theorem)}}
Let $R$ be a commutative simplicial ring.
Then $Br(R)\cong Br(R_0).$
\end{theorem}

The above triviality theorem leads us 
to consider a simplicial version of
the Brauer group \cite{HSZZ27}, denoted by $\overset{s}{Br}(R)=(\overset{s}{Br}(R)_n)_{n\geq 0}$, such that $\overset{s}{Br}(R)_0=
Br(R)$. Currently, the simplicial
group $\overset{s}{Br}(R)$ is still 
mysterious to us and not easy to compute. 

\subsection{Organization of the paper}
\label{yysec0.5}
In Section \ref{yysec1}, we review basic definitions of simplicial algebras and simplicial modules.
We define simplicial separability and illustrate the notion by examples over a constant base simplicial ring and simplicial matrix rings in Section \ref{yysec2}.
Centralizer and separable enriched functors are introduced in Section \ref{yysec3}, 
while more properties of simplicial separability are derived in Section \ref{yysec4}.
We review and adapt the enriched Morita theory for simplicial algebras in Section \ref{yysec5}.
Azumaya simplicial algebras are studied in Section \ref{yysec6},
and the simplicial Auslander-Goldman Theorem is proved.
We conclude the paper by a short discussion on the Brauer groups of commutative simplicial rings in Section \ref{yysec7}.

\section{Preliminaries}
\label{yysec1}
In this paper we try to
provide a comprehensive introduction 
of separable simplicial algebras with 
many details.

We start with reviewing 
basic definitions on simplicial algebras and simplicial modules.
A ring is always an associative unital ring,
and a ring homomorphism preserves the unit of a ring.
The category of rings is denoted by $\Ring$.
An algebra is always an associative unital algebra.

\subsection{Simplicial objects}
\label{yysec1.1}
Let $\Delta$ denote the simplex category 
whose objects are totally ordered sets 
$\{[n]\mid n=0,1,2,\cdots\}$ where 
$[n] := \{0, 1, \ldots, n\}$ and 
whose morphisms are all non-strict order-preserving functions 
from $[k]\longrightarrow [n]$ 
for every pair of non-negative integers $k$ and $n$. 
Let $\mathcal{C}$ be any 
category. 
A {\it simplicial object} $X$ in $\mathcal{C}$ is a 
contravariant functor from the simplex category $\Delta$ to ${\mathcal C}$.
A simplicial object can be described as follows. 
For each $n$, $X_n:=X([n])$ is an object in ${\mathcal C}$, 
and for any morphism $\theta: [k]\to [n]$ in $\Delta$, 
there is an induced map $\theta^*_{X}: X_n\to X_k$ which is a morphism in ${\mathcal C}$. 
The morphisms $\{\theta^*_X\}$ are called the {\it simplicial structure maps} of $X$. 
For each $n$, 
$X_n$ is called the {\it level $n$} part of $X$. 
We may simply write $X$ as $\{X_n\}_{n\geq 0}$ 
if there is no confusion about the simplicial structure maps. 
Simplicial objects can also be defined by using 
face and degeneracy maps \cite[Proposition 8.1.3]{Weibel1994}.
For $n\geq 1$,
the degeneracy map $s_0^n$ is induced by the unique map $[n] \to [0]$ in $\Delta$;
similarly, $d_0^n$ is induced by the unique map $[0]\to [n]$ mapping $0$ to $n$ in $\Delta.$
For each $n\geq 0$, 
we use $\Delta^n$ to denote the representable functor $\Hom_{\Delta}(-,[n])$, 
which is a simplicial set by definition.

\subsection{Simplicial rings}
\label{yysec1.2}

A {\it simplicial ring}, or an {\it sring}, is a simplicial object in $\Ring$.
A special example of an sring is $\Bbbk\Delta^0$;
the constant sring with an ordinary ring $\Bbbk$ on every level and  all the simplicial structure maps equal the identity maps.
In particular, $\Z \Delta^0$ is the constant sring with $\Z$ levelwise.

The category of simplicial abelian groups $
\sAb$ with $\sotimes_{\Z}$ (levelwise tensor product over $\Z$) is a symmetric monoidal category with unit $\Z\Delta^0$.
An sring $R$ is equivalently a monoid object in $(\sAb, \sotimes_{\Z}, \Z \Delta^0)$ \cite[p.~503]{SS2000}.
Namely,
there is a simplicial abelian group morphism (the multiplication map)
$\mu_{R/\Z\Delta^0}: R \sotimes_{\Z} R \to R$
and another simplicial abelian group morphism (the unit map)
$\eta_{R/\Z\Delta^0}: \Z \Delta^0 \to R$
such that $(R, \mu_{R/\Z\Delta^0}, \eta_{R/\Z\Delta^0})$ satisfies 
the usual associativity and unitality axioms.

The tensor product of srings is the tensor product of the underlying simplicial abelian groups.
Let $\sRing$ be the category of srings.
As the category of monoids in $(\sAb, \sotimes_{\Z}, \Z \Delta^0)$,
$\sRing$ is a symmetric monoidal category \cite[Appendix A.1]{Stonek2018}.

\begin{definition}
\label{yydef1.1}
    Let $R$ be an sring.
    \begin{enumerate}
        \item[(1)] A two-sided/left/right simplicial ideal, or {\it two-sided/left/right sideal}, $I$ of $R$ 
        is a simplicial subset $I\subseteq R$ 
        such that $I_n$ is a two-sided/left/right ideal of $R_n$ for every $n\geq 0.$

        \item[(2)] A  simplicial subring, or {\it sub sring},
        $R'$ is a simplicial subset $R'\subseteq R$ such that $R'$ is an sring.
    \end{enumerate}
    Unless otherwise specified,
    an sideal will be a two-sided sideal.
\end{definition}

From now on, we always assume that $R$ is an sring.
The following agree with the standard definitions in \cite{PareigisI}, \cite[II.6]{Quillen1967}, and \cite[p.~498]{SS2000}.

\begin{definition}
\label{yydef1.2}
    A {\it left simplicial module} over $R$, or a {\it left $R$-smodule}, is a left module object over the monoid $R$ in $(\sAb, \sotimes_{\Z}, \Z \Delta^0).$
    That is, a left $R$-smodule is a simplicial abelian group $M$ with a simplicial abelian group morphism
    $\mu_M:  R \sotimes_{\Z} M \to M$
    such that $(M, \mu_M)$ satisfies 
    the usual associativity
    and unitality axioms.
\end{definition}

\begin{definition}
\label{yydef1.3}
    Let $M$ and $N$ be left $R$-smodules.
    A {\it left $R$-smodule morphism} $f: M \to N$ is a simplicial abelian group morphism $f: M \to N$ such that, for every $r\in R_n$ and $m\in M_n,$
    $f(rm)=rf(m)$.
\end{definition}

\begin{remark}
\label{yyrem1.4}
Let $R$ be an sring.
\begin{enumerate}
\item[(1)] Let $M$ be a left $R$-smodule.
    Then we write $\mu_M(r \otimes m):= r\cdot_M m$ for any $r\otimes m\in (R\sotimes_{\Z} M)_n$ and $n\geq 0.$
    Often we ignore the dot when the context of the action is clear.
    
    \item[(2)] Right $R$-smodules and right $R$-smodule morphisms can be defined analogously.
    Let $R\sMod$ be the category of left $R$-smodules, 
    and $\sMod R$ be the category of right $R$-smodules.

    \item[(3)] Since an sring is a functor $\Delta^{op} \to \Ring$,
    postcomposing the opposite ring functor $(-)^{o}: \Ring \to \Ring$ induces 
    the {\it opposite sring} of $R$, denoted by $R^{o}:= \{R_n^{o}\}$ with simplicial structure maps $\theta^*_{R^{o}}:=(\theta^*_R)^{o}$  for every $\theta:[m] \to [n]$ in $\Delta$.

    \item[(4)] Note that $\sMod R$ is equivalent to $R^{o}\sMod$.
When $R$ is a commutative sring,
$\sMod R$ is equivalent to $R \sMod$.
In this case, we just say $R$-smodules without indicating whether the smodules are left or right,
and we denote the category of $R$-smodules by $R\sMod$.
\end{enumerate}
\end{remark}

\subsection{Simplicial algebras}
\label{yysec1.3}

The following are mostly standard definitions; see, for instance,  \cite[§VII]{EKMM1997} and \cite[p.~499]{SS2000}.

\begin{definition}
\label{yydef1.5}
Let $R$ be an sring.
The {\it center} of $R$ is 
$Z(R):=\{Z(R)_n\}$ where 
\begin{align*}
Z(R)_n:= \{x\in R_n
\mid 
\theta^*_R(x)\in Z(R_k), \forall \theta:[k] \to [n] \text{ in } \Delta\}.
\end{align*}
\end{definition}

It is easy to construct a nonseparable salgebra $A$ such that $Z(A)_n \neq  Z(A_n)$ for some $n.$ 

\begin{lemma}
\label{yylem1.6}
The center $Z(R)$ of an sring $R$ is a sub sring.
\end{lemma}

\begin{proof}
    Let $n\geq 0$.
    Let $x\in Z(R)_n$ and $\sigma:[m] \to [n]$ in $\Delta$ be arbitrary.
    We show $\sigma^*_R(x)\in Z(R)_m.$
    For any $\theta: [k]\to [m]$,
    \begin{align*}
        \theta^*_R \big( \sigma^*_R(x)\big) = (\sigma\theta)^*_R(x)\in Z(R_k)
    \end{align*}
    which implies $\sigma^*_R(x)\in Z(R)_m$.
    Hence $Z(R)$ is a simplicial subset of $R.$
    Moreover,
    $Z(R)_n = \bigcap_{\theta:[k] \to [n]} (\theta^*_R)^{-1}\big( Z(R_k) \big)$ is a subring of $R_n$.
    So $Z(R)$ is a sub sring.
\end{proof}

So $Z(R)$ is the largest sub sring that is levelwise contained in $Z(R_n)$.
An sring $(R,\mu_{R/\Z\Delta^0}, \eta_{R/ \Z\Delta^0})$ is {\it commutative} if $\mu_{R/\Z\Delta^0}$ is commutative.
It is clear that an sring $R$ is commutative if and only if $Z(R)=R.$

\begin{definition}
\label{yydef1.7}
Let $R$ be an  sring.
The {\it tensor product} of  a right $R$-smodule $M$ and a left $R$-smodule $N$ is $M \sotimes_{R} N$
where $(M\sotimes_{R} N)_n:= M_n \otimes_{R_n} N_n$
with simplicial structure maps $\theta^*_{M \sotimes_{R} N}$ defined by
$$ M_n \otimes_{R_n} N_n \to M_m \otimes_{R_m} N_m  \qquad x\otimes y \mapsto  \theta^*_{M}(x) \otimes \theta^*_N(y)$$ 
for every $x\otimes y \in M_n \otimes_{R_n} N_n$ and every $\theta:[m]\to [n]$ in $\Delta.$
\end{definition}

\begin{remark}
\label{yyrem1.8}
Let $R$ be an sring, not necessarily commutative.
\begin{enumerate}
\item[(1)] 
When $R=\Z\Delta^0$,
the tensor product $\sotimes_{\Z\Delta^0}$ is exactly the tensor product $\sotimes_{\Z}$ in $\sAb.$
        So any simplicial abelian group is equivalently a $\Z\Delta^0$-smodule.
        When $R=\Bbbk\Delta^0$ for a commutative ring $\Bbbk,$
        we write $\sotimes_{\Bbbk}:=\sotimes_{\Bbbk\Delta^0}.$

        \item[(2)] As an example,
        $R\Delta^n : =R \sotimes_{\Z} \Z \Delta^n$ is naturally a left $R$-smodule, 
        where
        $\Z\Delta^n$ is the free simplicial abelian group generated by $\Delta^n$.

        \item[(3)] We note $R\sMod$ is abelian and exactness is defined levelwise \cite[§II.6]{Quillen1967}.
        The kernels and cokernels of an $R$-smodule morphism are computed levelwise.
        In particular, an $R$-smodule morphism is an epimorphism if and only if it is levelwise surjective. 

        \item[(4)] When $R$ is a commutative sring,
        $(R \sMod, \sotimes_{R}, R)$
        is a symmetric monoidal category \cite[p.~499]{SS2000}.

    \end{enumerate}
\end{remark}

\begin{definition}
\label{yydef1.9}
    Let $R$ be a commutative sring. 
    An sring $A$ is an {\it $R$-salgebra} if 
    there is an sring morphism $\eta_{A/R}: R \to A $ such that
    $\eta_{A/R}(R) \subseteq Z(A).$
    An $R$-salgebra morphism $f:A \to B$ is an sring morphism such that $f\circ \eta_{A/R} = \eta_{B/R}.$
\end{definition}

\begin{remark}
\label{yyrem1.10}
    Let $R$ be a commutative sring.
    \begin{enumerate}
        \item[(1)] Equivalently, an $R$-salgebra $A$ is a monoid object in the category of $R$-smodules \cite[p.~499]{SS2000}.
        In this way,
        the tensor product of $R$-salgebras is
        the tensor product of the underlying $R$-smodules. 
   
        \item[(2)] When $\Bbbk$ is a commutative ring, 
        a simplicial $\Bbbk$-algebra  refers
        to a simplicial object in the category of $\Bbbk$-algebras,
        which is equivalently
        a monoid object in the category of simplicial $\Bbbk$-modules \cite{SS2003}.
        A simplicial $\Bbbk$-algebra is equivalently a $\Bbbk\Delta^0$-salgebra, which will be called a $\kk$-salgebra throughout.

        \item[(3)]  If $A$ is an $R$-salgebra,
        a left $A$-smodule is equivalently an $R$-smodule  $M$ with $R$-smodule morphism
        $A \sotimes_R M \to M $
        satisfying the associativity and unitality of a left action \cite[§VII, Lemma 1.8]{EKMM1997}.
    \end{enumerate}
\end{remark}

\section{Separability of Simplicial Algebras}
\label{yysec2}
We define separable simplicial algebras over a commutative sring, characterize separable simplicial algebras over a constant commutative ring, and investigate the separability of the simplicial matrix rings as an example of simplicial algebras over a non-constant commutative sring.

\subsection{A few projectivity lemmas}
\label{yysec2.1}

\begin{lemma}
\label{yylem2.1}
    Let $R$ be an sring and let $I$ be a left sideal of $R$.
    Suppose $x\in R_0$.
    Define
    $(Ix)_n= I_ns_0^n(x)$.
    Then $Ix:=\{(Ix)_n\}_{n\geq 0}$ is a left sideal of $R$.
\end{lemma}

\begin{proof}
    Let $\theta:[m] \to [n]$ be in $\Delta.$
    Note
    \begin{align*}
        \theta^*_R\big( (Ix)_n \big) = \theta^*_R\big( I_ns_0^n(x) \big) = \theta^*_R(I_n)s_0^m(x) \subseteq I_ms_0^m(x) = (Ix)_m.
    \end{align*}
    This implies $Ix$ is closed under the simplicial structure maps of $R$ and hence a simplicial subset.
    Since $I_ns_0^n(x)$ is a left ideal of $R_n$,
    $Ix$ is a left sideal of $R$.
\end{proof}

\begin{lemma}
\label{yylem2.2}
Let $R$ be an sring and $e\in R_0$ be an idempotent. 
Then 
\begin{enumerate}
\item[(1)] $s_0^n(e)$ is an idempotent in $R_n$ for every $n\geq 0$.

        \item[(2)] View $R$ as an sideal of $R$.
        Then $Re$ is a left sideal of $R.$

        \item[(3)] As left $R$-smodules, $R = Re \oplus R(1_{R_0}-e).$
    \end{enumerate}
\end{lemma}

\begin{proof}
    (1): Immediate since $s_0^n: R_0 \to R_n$ is a ring homomorphism.

    (2): By Lemma \ref{yylem2.1}.

    (3): It is clear that $1_{R_0}-e\in R_0$ is an idempotent.
    So by part (2), both $Re$ and $R(1_{R_0}-e)$ are left sideals and,
    in particular, left $R$-smodules.
    The rest follows from the decomposition $R_n = R_ns_0^n(e) \oplus R_n(1_{R_n}-s_0^n(e))$ for every $n\geq 0$.
\end{proof}

\begin{definition}
\label{yydef2.3}
Let $P$ be a left $R$-smodule.
We say $P$ is {\it projective} if the functor $\Hom_{R \sMod}(P,-): R \sMod \to \Ab$ is exact.
\end{definition}

The projectivity of the unit in a monoidal category is useful in the study of separability of monoids; see \cite[Proposition 4.5]{BT2015} and \cite[p.~98 and Theorem 4.14]{VOZ1998}.

\begin{lemma}
\label{yylem2.4}
    Let $R$ be an sring.
    Then $R$ is a projective left $R$-smodule.
\end{lemma}

\begin{proof}
    Any left $R$-smodule morphism $f: R\to M$ is uniquely determined by $f(1_{R_0})\in M_0$, 
    since $f(r)=rf(1_{R_n})=rs_0^n(f(1_{R_0}))$ for every $r\in R_n$.
    So any left $R$-smodule surjection $M\to R$ splits by looking at a preimage of $1_{R_0}$.
\end{proof}

The following is a generalization of \cite[Exercise 1.1.22]{Ford2017} in the simplicial setting.

\begin{proposition}
\label{yypro2.5}
    Let $R$ be an sring and $I$ be a left sideal of $R$.
    The following are equivalent.
    \begin{enumerate}
        \item[(1)] $R/I$ is a projective left $R$-smodule.

        \item[(2)] The sequence of left $R$-smodules $0\to I \to R \to R/I \to 0$ is split exact.

        \item[(3)] As a left $R$-smodule,
        $I$ is a direct summand of $R$.
        
        \item[(4)] There is an element $e\in R_0$ such that $1_{R_0}-e\in I_0$ and $Ie=0$.

        \item[(5)] There is an idempotent $e\in R_0$ such that $I=R(1_{R_0}-e)$.
    \end{enumerate}
\end{proposition}

\begin{proof}
    (1) $\implies$ (2) $\implies $ (3): Immediate.

    (3) $\implies$ (4): Write $R = I \oplus N$ as left $R$-smodules.
    Then $1_{R_0}= (1_{R_0}-e)+ e$ for some $e\in N_0$.
    It follows that $Ie = I\big(1_{R_0}-(1_{R_0}-e) \big) \subseteq I + I(1_{R_0}-e)  \subseteq I$.
    Moreover, $Ie\subseteq IN\subseteq N$ since $N$ is a left $R$-smodule.
    Hence $Ie \subseteq I \cap N =0.$

    (4) $\implies$ (5): Since $Ie=0$, $(1_{R_0}-e)e=0$ implies $e$ is an idempotent.
    Hence $R=Re \oplus R(1_{R_0}-e)$ as left $R$-smodules by Lemma \ref{yylem2.2}(3).
    Since $I$ is a left sideal,
    $R(1_{R_0}-e) \subseteq I.$
    On the other hand,
    note that $$I=I(1_{R_0} -e + e)\subseteq 
    I(1_{R_0}-e)+Ie \subseteq I(1_{R_0}-e)\subseteq R(1_{R_0}-e).$$
    This completes the proof of $I = R(1_{R_0}-e).$

    (5) $\implies $ (1): Since $e\in R_0$ is an idempotent,
    $R= Re\oplus R(1_{R_0}-e)= Re\oplus I$ as left $R$-smodules by Lemma \ref{yylem2.2}(3).
    It follows that $R/I\cong Re$ is projective as a direct summand of the projective left $R$-smodule $R$ (Lemma \ref{yylem2.4}).
\end{proof}

From now on,
we always assume that $R$ is a commutative sring.

\begin{definition}
\label{yydef2.6}
Let $A$ be an $R$-salgebra.
The {\it enveloping algebra} of $A$ over $R$ is $A^e := A \sotimes_R A^{o}$.
Then $A$ is a left $A^e$-smodule under the action 
$$(a\otimes b)\cdot_A x = axb$$
for every $a,b,x\in A_n$ extended additively.
Therefore,
there is a surjective left $A^e$-smodule morphism
$$\mu_{A/R}:= A^e \to A \qquad u \mapsto u \cdot_A 1_{A_n}$$
for any $u\in (A^e)_n$.
Define $J_{A/R}:= \ker \mu_{A/R}$, which is a left sideal of $A^e.$
For every $k\geq 0,$
it is easy to verify that $(J_{A/R})_k= \sum_{a\in A_k} (A^e)_k(a \otimes 1_{A_k} - 1_{A_k} \otimes a)$.
\end{definition}

The following is a generalization of \cite[Proposition 4.1.2]{Ford2017}.
A more general version for braided monoidal categories appears in \cite[Proposition 4.1]{VOZ1998}.

\begin{proposition}
\label{yypro2.7}
    Let $A$ be an $R$-salgebra.
    The following are equivalent.
    \begin{enumerate}
        \item[(1)] $A$ is a projective left $A^e$-smodule.

        \item[(2)] The exact sequence of left $A^e$-smodules $0 \to J_{A/R} \to A^e \xrightarrow{\mu_{A/R}} A \to 0$ splits.

        \item[(3)] There is an element $e\in (A^e)_0$  such that $\mu_{A/R}(e)=1_{A_0}$ and $J_{A/R}e=0.$

        \item[(4)] There is an idempotent $e\in (A^e)_0$ such that $J_{A/R}= A^e(1_{(A^e)_0}-e).$
    \end{enumerate}
\end{proposition}

\begin{proof}
Apply Proposition \ref{yypro2.5} to $(A^e, J_{A/R} )$. Then the rest follows from the facts that $A\cong A^e/J_{A/R}$ and that $\mu(e)=1_{A_0}$ if and only if $1_{(A^e)_0}-e\in (J_{A/R})_0.$
\end{proof}

The following notion of separability of simplicial algebras agrees with the well-studied notion of separability of monoid objects in monoidal categories  \cite{ BT2015, VOZ1998, PareigisIV}.

\begin{definition}
\label{yydef2.8}
    Let $A$ be an $R$-salgebra.
    Then $A$ is {\it separable over $R$}, or a {\it separable $R$-salgebra},
    if any one of the equivalent conditions in Proposition \ref{yypro2.7} is satisfied.
    An element $e\in (A^e)_0$ such that $\mu_{A/R}(e)=1_{A_0}$ and
    $$(1_{A_n}\otimes a)s_0^n(e) = (a \otimes 1_{A_n})s_0^n(e)$$ for every $a\in A_n$
    is called a {\it separability idempotent}.
\end{definition}

\begin{remark}
\label{yyrem2.9}
Let $A$ be a separable $R$-salgebra.
    \begin{enumerate}
        \item[(1)] If $A$ is commutative,
        then the separability idempotent is unique.

        \item[(2)] If $R=\Bbbk\Delta^0$ for a commutative ring $\Bbbk,$
        we say $A$ is separable over $\Bbbk.$

        \item[(3)] Let $e\in (A^e)_0$ be a separability idempotent of $A.$
        Then $s_0^n(e)$ is a separability idempotent of the $R_n$-algebra $A_n$.
        So
        $A_n$ is separable over $R_n.$
    \end{enumerate}
\end{remark}

\subsection{Separable simplicial algebras over a constant simplicial ring}
\label{yysec2.2}
We characterize the separability of simplicial algebras over a constant commutative simplicial ring by a rigidity result.

\begin{lemma}
\label{yylem2.10}
        Let $\Bbbk$ be a commutative ring and $A$ be a separable $\Bbbk$-salgebra.
        Let $e\in (A^e)_0$ be a separability idempotent.
        Let $n\geq 0$.
        If $f: A_n \to A_n$ is an idempotent $\Bbbk$-algebra morphism such that $(f\otimes f^{o})(s_0^n(e) )=s_0^n(e)$,
        then $f=Id_{A_n}.$
    \end{lemma}

    \begin{proof}
         Since $f\otimes f^{o}$ fixes $s_0^n(e)$ and $f$ is idempotent, 
        \begin{align}
        \label{E2.10.1}
        \tag{E2.10.1}
            (f\otimes Id_{A^{o}_n}) (s_0^n(e)) &= (f\otimes Id_{A^{o}_n}) (f\otimes f^{o})(s_0^n(e)) = (f\otimes f^{o})(s_0^n(e)) = s_0^n(e).
        \end{align}
        For any $a\in A_n$, since $e$ is a separability idempotent,
        \begin{align}
        \label{E2.10.2}
        \tag{E2.10.2}
           (a\otimes 1_{A_n})s_0^n(e) = (1_{A_n} \otimes a) s_0^n(e).
        \end{align}
        Applying the $\Bbbk$-algebra morphism $f\otimes Id_{A^{o}_n}$ to \eqref{E2.10.2}, 
        \begin{align}
        \label{E2.10.3}
        \tag{E2.10.3}
          (f(a) \otimes 1_{A_n})s_0^n(e)
          &=(1_{A_n}\otimes a) s_0^n(e)=(a\otimes 1_{A_n})s_0^n(e).
        \end{align}
        Then apply the left $A^e$-smodule morphism 
        $\mu_{A/\Bbbk \Delta^0}$ to \eqref{E2.10.3}, 
        \begin{align*}
            f(a) = f(a)\mu_{A/\Bbbk \Delta^0}(s_0^n(e)) = a\mu_{A/\Bbbk \Delta^0}(s_0^n(e))=a
        \end{align*}
        since $\mu_{A/\Bbbk \Delta^0}(e)=1_{A_0}.$
        Therefore,
        $f=Id_{A_n}.$
    \end{proof}

\begin{theorem}
\label{yythm2.11}
    Let $\Bbbk$ be a commutative ring and $A$ be a $\Bbbk$-salgebra.
    Then $A$ is separable over $\Bbbk$ if and only if $A\cong A_0\Delta^0 :=\Bbbk\Delta^0 \sotimes_{\Bbbk} A_0$ (the constant sring determined by $A_0$) and $A_0$ is separable over $\Bbbk$.
\end{theorem}

\begin{proof}
    ($\implies$) Suppose $A$ is separable over $\Bbbk$ and let $e\in (A^e)_0$ be a separability idempotent.
    From the canonical $\Bbbk$-salgebra inclusion $A_0\Delta^0 \hookrightarrow A$,
    for every $n\geq 0,$
    there is a $\Bbbk$-algebra morphism 
    $s_0^n: A_0 \to A_n.$
    Note that $d_0^ns_0^n = Id_{A_0}$ which implies 
    $s_0^nd_0^n: A_n \to A_n$ is idempotent.
    Moreover,    \begin{align*}
        \big( (s_0^nd_0^n) \otimes (s_0^nd_0^n) \big) s_0^n(e) = \big( (s_0^nd_0^n) \otimes (s_0^nd_0^n) \big) (s_0^n \otimes s_0^n)(e) = s_0^n(e).
    \end{align*}
    By Lemma \ref{yylem2.10}, $s_0^nd_0^n = Id_{A_n}$ and hence $s_0^n$ is an isomorphism.
    Thus, $A_0\Delta^0\hookrightarrow A$ is an isomorphism of $\Bbbk$-salgebras.
    Finally, 
    we note the separability idempotent $e$ makes $A_0$ separable over $\Bbbk$ as in Remark \ref{yyrem2.9}(3).

    ($\impliedby$) Suppose $A_0$ is separable over $\Bbbk$.
    Any separability idempotent $e \in (A_0)^e$ clearly makes the constant $\Bbbk$-salgebra $A_0\Delta^0$ separable.
\end{proof}

\begin{corollary}
\label{yycor2.12}
Let $\Bbbk$ be a field and let $A$ be a separable simplicial algebra over $\Bbbk$. Then $A\sMod$ has infinite
global dimension.
\end{corollary}

\begin{proof}
By Theorem \ref{yythm2.11},
$A\cong\Bbbk\Delta^0 \sotimes_{\Bbbk} A_0$ where $A_0$ is a separable algebar
over $\Bbbk$.
By Dold-Kan Correspondence \cite[Theorem 3.1]{HSZZ26},
$A\sMod\cong {\rm{Ch}}_{\geq 0}(A_0\Mod)$. So it suffices to show that
${\rm{Ch}}_{\geq 0}(A_0\Mod)$ has infinite global dimension. Let $M$ be 
a simple left $A_0$-module. Consider $M$ as an object in ${\rm{Ch}}_{\geq 0}(A_0\Mod)$, it 
has a minimal injective resolution
$$0\to M \to I^{0}
\to I^{1}\to \cdots \to I_n\to \cdots $$
where $I^n=\Sigma^n(\Bbbk\to\Bbbk)\otimes M$, for all $n\geq 0$, 
are injective objects in ${\rm{Ch}}_{\geq 0}(A_0\Mod)$ by \cite[Lemma A.3(5)]{HSZZ26}.
Hence $M$ has infinite injective dimension.
Therefore $A\sMod$ has infinite global
dimension.
\end{proof}

\subsection{Separable simplicial algebras over a non-constant simplicial ring}
\label{yysec2.3}

Separability of a simplicial algebra depends on the base sring.
When the base sring is not constant, 
the rigidity in Theorem \ref{yythm2.11} fails.

\begin{definition}
\label{yydef2.13}
Let $\Bbbk$ be a commutative ring and $A: \Delta^{op}\to \kAlg$ be a $\Bbbk$-salgebra.
Let $k\geq 1.$
The {\it simplicial matrix ring} $M_k(A)$ is 
\begin{align*}
    M_k(A):=\Delta^{op} \xrightarrow[]{A} \kAlg \xrightarrow[]{M_k(-)} \kAlg,
\end{align*}
where 
$M_k(-)$ is the functor of taking the $k\times k$ matrix ring of an ordinary $\Bbbk$-algebra.
In particular,
$M_k(A)_m := M_k(A_m)$ and
for every $\theta:[p]\to [m]$ in $\Delta,$
the induced simplicial structure map $\theta^*_{M_k(A)}: M_k(A)_m \to M_k(A)_p$ sends $(a_{ij})$ to $\big(\theta^*_A(a_{ij}) \big)$.
\end{definition}

For every $k\geq 1$, it is clear from Definition \ref{yydef2.13} that $M_k(A)$ is a $\Bbbk$-salgebra.
Furthermore, there is a canonical sring injection
\begin{align*}
    \delta: A \to M_k(A) \qquad a \mapsto aI_k.
\end{align*}
Note that $\delta(A)\subseteq Z\big(M_k(A) \big)$ if and only if
$A$ is commutative.
Therefore,
when $A$ is commutative, $M_k(A)$ is naturally an $A$-salgebra.
This leads to two questions:
\begin{enumerate}
    \item[(a)] View $A$ as a $\Bbbk$-salgebra.
    When is $M_k(A)$ separable over $\Bbbk$?
    
    \item[(b)] Suppose $A$ is a commutative sring.
    When is $M_k(A)$ separable over $A$?
\end{enumerate}

The following lemma is well-known.

\begin{lemma}
\label{yylem2.14}
Let $\Bbbk$ be a commutative ring and $B$ be a $\Bbbk$-algebra.
For every $k\geq 1$, $M_k(B)$ is separable over $\Bbbk$ if and only if $B$ is separable over $\Bbbk.$
\end{lemma}

Question (a) can be answered by a simplicial version of Lemma \ref{yylem2.14}.

\begin{proposition}
\label{yypro2.15}
    Let $\Bbbk$ be a commutative ring and $A$ be a $\Bbbk$-salgebra.
    For every $k\geq 1$,
    the following are equivalent.
    \begin{enumerate}
        \item[(1)] $M_k(A)$ is separable over $\Bbbk$.

        \item[(2)] $M_k(A) \cong M_k(A_0)\Delta^0$ and $M_k(A_0)$ is separable over $\Bbbk$.

        \item[(3)]
        $A\cong A_0\Delta^0$ and $A_0$ is separable over $\Bbbk$.

        \item[(4)] $A$ is separable over $\Bbbk$.
    \end{enumerate}
\end{proposition}
\begin{proof}
    (1) $\iff$ (2): By Theorem \ref{yythm2.11}.

    (2) $\iff$ (3): It is clear that $M_k(A)$ is a constant sring if and only if $A$ is a constant sring.
    Then
    by Lemma \ref{yylem2.14}, $M_k(A_0)$ is separable over $\Bbbk$ if and only if $A_0$ is separable over $\Bbbk$.

    (3) $\iff$ (4): By Theorem \ref{yythm2.11}.
\end{proof}

If $\Bbbk$ is a commutative ring, then $M_k(\Bbbk)$ is separable over $\Bbbk$ \cite[Example 4.2.1]{Ford2017}.
Question (b) can be answered by a simplicial generalization of this fact.

\begin{proposition}
\label{yypro2.16}
    Suppose $A$ is a commutative sring.
    For every $k\geq 1$,
    $M_k(A)$ is separable over $A$.
\end{proposition}

\begin{proof}
Let $1 \leq j \leq k$ be fixed.
By abuse of notation, let $E_{pq}$ be the $k \times k$ matrix with $1_{A_n}$ in the $(p,q)$-entry and 0 elsewhere and $I_k$ be the identity matrix in each level $n$ of $M_k(A)$. 
Let $M_k(A)^e= M_k(A) \sotimes_A M_k(A)^{o}.$
Define $e=\sum_{i=1}^{k} E_{ij} \otimes_{A_0} E_{ji} \in M_k(A_0)^e$ and then, by abuse of notation, $s_0^m(e)=\sum_{i=1}^{k} E_{ij} \otimes_{A_m} E_{ji}.$

It is clear that 
$\mu_{M_k(A)/A}(e)=\sum_{i=1}^{k}E_{ii}=I_k= 1_{M_k(A_0)}.$
By Proposition \ref{yypro2.7},
it remains to show that $J_{M_k(A)/A}e=0$.
Recall, for every $m\geq 0$,
\begin{align*}
    (J_{M_k(A)/A})_m=\sum_{E\in M_k(A_m)} M_k(A_m)^e(E\otimes_{A_m} I_k - I_k\otimes_{A_m} E)
\end{align*} 
for every $E=(a_{pq})\in M_k(A_m)$.
Since $E=\sum a_{pq}E_{pq}=\sum E_{pq}a_{pq}$,
by $A_m$-linearity,
\begin{align*}
     a_{pq} E_{pq} \otimes_{A_m} I_k - I_k \otimes_{A_m} a_{pq}E_{pq} = (a_{pq}I_k \otimes_{A_m} I_k)(E_{pq} \otimes_{A_m} I_k - I_k \otimes_{A_m} E_{pq}).
\end{align*}
It suffices to note that, for every $m\geq 0,$
\begin{align*}
    (E_{pq} \otimes_{A_m} I_k)s_0^m(e) &=\sum_{i=1}^k E_{pq}E_{ij}  \otimes_{A_m} E_{ji} = E_{pj} \otimes_{A_m} E_{jq}\\
    &= \sum_{i=1}^k E_{ij}\otimes_{A_m} E_{ji}E_{pq}=(I_k \otimes_{A_m} E_{pq})s_0^m(e)
\end{align*}
which shows $(J_{M_k(A)/A})_m s_0^m(e)=0$.
So $e$ is a separability idempotent of $M_k(A)$ and $M_k(A)$ is separable over $A$.
\end{proof}

\begin{remark}
\label{yyrem2.17}
    When $A$ is a commutative sring,
    Proposition \ref{yypro2.16} shows the following. \begin{enumerate}
        \item[(1)] 
        A separable simplicial algebra over a
        non-constant sring need not be a constant sring, in contrast to Theorem \ref{yythm2.11}.
        Indeed, $M_k(A)$ is always separable over $A$ (Proposition \ref{yypro2.16}),
        but the underlying sring is not a constant sring when $A$ is not a constant sring.
    
    \item[(2)] The separability idempotent of the noncommutative $A$-salgebra $M_k(A)$ is not unique when $k\geq 2$,  
    in contrast to Remark \ref{yyrem2.9}(1).
    \end{enumerate}
\end{remark}

\section{Separability in Simplicial Enrichment}
\label{yysec3}

We follow \cite{Kelly1982} for the general theory of enriched categories.
Let $R$ be a commutative sring. 
Then $R \sMod$ is a closed symmetric monoidal category.
It follows that $A \sMod$ is tensored, cotensored, and
$R \sMod$-enriched for any $R$-salgebra $A$.
In this enriched setting,
we study the notion of separability through the simplicial centralizer and the enriched forgetful functor $\mathcal{U}_{A/R}:A \sMod \to R\sMod$.

\subsection{Basics of simplicial enrichment}
\label{yysec3.1}

Let $R$ be a commutative sring and 
recall the $R$-smodule $R\Delta^n = R\sotimes_{\Z} \Z\Delta^n.$
The following constructions of a simplicial Hom-set in $R \sMod$ agree with \cite[II.6]{Quillen1967}.

\begin{definition}
\label{yydef3.1}
    For any $R$-smodules $M, N$,
    we define the {\it simplicial Hom-set} $\shom_{R}(M,N)$ to be
    \begin{align*}
        \shom_{R}(M,N)_n:=\Hom_{R \sMod}(R\Delta^n \sotimes_R M, N)
    \end{align*}
    where each $\theta:[k]\to [n]$ in $\Delta$ induces $\theta^*:\shom_{R}(M,N)_n \to \shom_{R}(M,N)_k$ sending every $f:R\Delta^n \sotimes_R M \to N$ to
    \begin{align}
    \label{E3.1.1}
    \tag{E3.1.1}
         \theta^*(f): R\Delta^k \sotimes_R M \xrightarrow[]{(\theta \circ -) \sotimes_R Id_M }R\Delta^n \sotimes_R M\xrightarrow[]{f} N. 
    \end{align}
\end{definition}

\begin{lemma}
\label{yylem3.2}
    For any $R$-smodules $M$ and $N$,
    $\shom_{R}(M,N)$ is an $R$-smodule.
\end{lemma}

\begin{proof}
    It is clear that $\shom_R(M,N)$ is a simplicial abelian group.
    We specify the $R$-action on $\shom_{R}(M,N)$ as follows:
    \begin{align*}
        \mu: R \sotimes_{\Z} \shom_R(M,N) \to \shom_R(M,N) \qquad r\otimes f \mapsto r\cdot f 
    \end{align*}
    for any $r \otimes f\in R_n \otimes_{\Z} \shom_R(M,N)_n$ where
    for every $\sigma \in \Delta^n_k$ and $x\in M_k$,
    \begin{align}
    \label{E3.2.1}
    \tag{E3.2.1}
        (r \cdot f)(\sigma \otimes x) := \sigma^*_R(r)\cdot_N f(\sigma \otimes x).
    \end{align}
    By \eqref{E3.2.1} and \eqref{E3.1.1},
    one can verify that $\theta^*(r\cdot f) =\theta^*_R(r) \cdot \theta^*(f)$ for every $\theta:[k] \to [n]$ in $\Delta.$
    Therefore, $\mu$ is a well-defined simplicial map.
    Moreover, $\mu$ satisfies the associativity and unitality in Definition \ref{yydef1.2} as $N$ is an $R$-smodule.
\end{proof}

\begin{remark}
\label{yyrem3.3}
    Retain the notation in Definition \ref{yydef3.1}.
    
    \begin{enumerate}
        \item[(1)] We call a morphism in $\shom_R(M,N)_n$ an {\it $n$-morphism},
        whereas
        a morphism in $\shom_{R}(M,N)_0=\Hom_{R\sMod}(M,N)$ is called a {\it (0-)morphism}.

        \item[(2)] Recall that $R\sMod$ is an abelian category (Remark \ref{yyrem1.8}(3)).
        When we refer to a 
        short exact sequence in $R\sMod$, 
        we always 
        use $0$-morphisms.
    \end{enumerate}
\end{remark}

\begin{definition}
\label{yydef3.4}
    Let $M,N,$ and $L$ be $R$-smodules.
    Define the composition
    \begin{align}
    \label{E3.4.1}
    \tag{E3.4.1}
         \circ  : \shom_R(N,L) \sotimes_R \shom_{R}(M,N) \to \shom_R(M,L)
    \end{align}
    where for every $f\in \shom_R(N,L)_n$ and $g\in \shom_{R}(M,N)_n$,
    their composition is
    \begin{align}
    \label{E3.4.2}
    \tag{E3.4.2}
        f\circ g:= R\Delta^n \sotimes_R M \xrightarrow[]{diag \sotimes_R Id_M } R\Delta^n \sotimes_R R\Delta^n \sotimes_R M \xrightarrow[]{Id_{R\Delta^n} \sotimes_R g} R\Delta^n \sotimes_R N \xrightarrow[]{f} L.
    \end{align}
    as a composition of $R$-smodule ($0$-)morphisms.
    For any $\sigma\in \Delta^n_k$ and $x\in M_k$,
    \begin{align}
    \label{E3.4.3}
    \tag{E3.4.3}
        (f\circ g)(\sigma \otimes x)=f\big( \sigma \otimes g(\sigma \otimes x)\big).
    \end{align}
    By \eqref{E3.2.1} and \eqref{E3.4.3}, 
    the composition $\circ$ also respects the $R$-action
    as
    \begin{align}
    \label{E3.4.4}
    \tag{E3.4.4}
        (r\cdot f) \circ g = f\circ (r\cdot g) = r\cdot (f\circ g)
    \end{align}
    for any $r\in R_n$.
    From \eqref{E3.4.3},
    it is easy to verify $\theta^*(f\circ g)=\theta^*(f)\circ \theta^*(g)$ which means $\circ$ is a well-defined simplicial map.
    It is also easy to verify $\circ$ is associative.
\end{definition}

\begin{remark}
\label{yyrem3.5}
Let $R$ be a commutative sring.
    \begin{enumerate}
        \item[(1)] For every $R$-smodule $M$, 
        the {\it identity $n$-morphism} on $M$ is
        \begin{align}
        \label{E3.5.1}
        \tag{E3.5.1}
            s_0^n(Id_M)(\sigma \otimes x)=x
        \end{align}
        for every $\sigma \in \Delta^n_k$ and $x\in M_k$.
        One can verify that $s_0^n(Id_M)$ is the unit with respect to the composition $\circ$ given in \eqref{E3.4.1}.

        \item[(2)] For any $R$-smodule $M$, 
        we note $\big(\shom_R(M,M), \circ\big)$ is an sring 
        and an $R$-salgebra via the sring morphism (due to \eqref{E3.4.4})
        \begin{align*}
            R\to \shom_R(M,M) \qquad r\mapsto r\cdot s_0^n(Id_M).
        \end{align*}
        Then, with respect to $\circ$, we  note that $\shom_R(M,N)$ is a right $\shom_R(M,M)$-smodule, 
        whereas $\shom_R(P,M)$ is a left $\shom_R(M,M)$-smodule for every $R$-smodules $P$ and $N$.

        \item[(3)] The category $(R \sMod, \sotimes_R, R)$ is a closed symmetric monoidal category (\cite[Theorem 4.1(2)]{SS2000} or \cite[Lemma 2.2.8]{HSS2000}).

        \item[(4)] Let $A$ be any $R$-salgebra.
        The simplicial Hom-set of left $A$-smodules is
        $\shom_{A}(M,N)_n = \Hom_{A\sMod} (R\Delta^n \sotimes_R M , N)$ where $R\Delta^n \sotimes_R M $ admits the left $A$-action on the $M$-component.
    \end{enumerate}
\end{remark}

Recall the notion of a {\it tensored} and {\it cotensored} enriched category \cite[p.~38]{Riehl2014}.

\begin{theorem} 
\label{yythm3.6}
    Let $R$ be a commutative sring. 
    Suppose $A$ is an $R$-salgebra.
    Then $A\sMod$ is tensored, cotensored, and $R\sMod$-enriched.
    In particular, there are isomorphisms of $R$-smodules
    \begin{align*}
        \shom_A(V \sotimes_R M, N) \cong \shom_R(V, \shom_A(M,N)) \cong \shom_A(M, \shom_R(V,N))
    \end{align*}
    that are enriched natural for every $R$-smodule $V$ and left $A$-smodules $M,N$.
\end{theorem}

\begin{proof}
    This follows from \cite[p.~395-396]{LM2007} and is implicit in \cite[Section 4.1]{Hovey1999}.
\end{proof}

\subsection{Centralizers}
\label{yysec3.2}

We extend the classical notion of centralizers \cite[(4.5)]{Ford2017} to the $R\sMod$-enriched setting and use it to characterize separability.

\begin{definition}
\label{yydef3.7}
    Let $A$ be an $R$-salgebra
    and $M$ be a left $A^e$-smodule.
    We define the {\it centralizer of $A$ in $M$} to be $M^A$, where
    \begin{align}
    \label{E3.7.1}
    \tag{E3.7.1}
       (M^A)_n:= \{x\in M_n \mid (J_{A/R})_k\theta^*_M(x) = 0, \forall \theta: [k] \to [n] \text{ in } \Delta\}
    \end{align}
    for every $n\geq 0$ and $J_{A/R}= \ker (\mu_{A/R}: A^e  \to A)$ as in Definition \ref{yydef2.6}.
\end{definition}

\begin{remark}
\label{yyrem3.8}
Let $A$ be an $R$-salgebra.
\begin{enumerate}

    \item[(1)] Any left $A^e$-smodule is equivalently an $(A,A)$-sbimodule.    
    Consequently,
    the centralizer of $A$ in $M$ can also be identified with
    \begin{align*}
        (M^A)_n:= \{x\in M_n \mid a \theta^*_M(x) =\theta^*_M(x)a, \forall \theta:[k] \to [n] \text{ in $\Delta$ and } a\in A_k\}.
    \end{align*}

    \item[(2)] Therefore, $A^A=Z(A)$ when $A$ is viewed as a left $A^e$-smodule.
    Similarly,
    if $I$ is an sideal of $A,$
    then $(A/I)^A=Z(A/I).$
\end{enumerate}
    
\end{remark}

The following generalizes \cite[Lemma 4.1.4]{Ford2017} in $R\sMod$-enriched setting.

\begin{lemma}
\label{yylem3.9}
    Let $A$ be an $R$-salgebra.
    Let $M$ be a left $A^e$-smodule.
    \begin{enumerate}
        \item[(1)] $M^A$ is an $R$-smodule.

        \item[(2)] $(-)^A: A^e \sMod \to R \sMod$ is an $R \sMod$-enriched functor.

        \item[(3)] $(-)^A \cong \shom_{A^e}(A,-)$ as  $R \sMod$-enriched functors.
    \end{enumerate}
\end{lemma}
\begin{proof}
    (1):  Note that any $r\in R_n$ acts on $M$ via $r\otimes 1_{A_n}$ and equivalently $1_{A_n}\otimes r$.
    Since the $R$-action commutes with the $A^e$-action, 
    the $R$-action commutes with the $(J_{A/R})_n$-action.
    Therefore,
    each $(M^A)_n$ is an $R_n$-module.

    So it suffices to show that $M^A$ is a simplicial subset of $M.$
    Let $x\in (M^A)_n$ and $\sigma:[m] \to [n]$ in $\Delta$ be arbitrary.
    For any $\theta:[k]\to [m]$ in $\Delta,$
    by \eqref{E3.7.1},
    \begin{align*}
        (J_{A/R})_k\theta^*_M\big( \sigma^*_M(x)\big)= (J_{A/R})_k(\sigma\theta)^*_M(x)= 0
    \end{align*}
    which implies $\sigma^*_M(x)\in (M^A)_m$ and $M^A$ is thus a simplicial subset of $M.$

    (2): Let $f\in \shom_{A^e}(M,N)_n$.
    It is clear that $(R\Delta^n \sotimes_R M)^A = R\Delta^n \sotimes_R M^A$ since $A^e$ acts solely on the $M$-component.
    We claim $f$ restricts to a morphism in $\shom_{R}(M^A, N^A)_n$.
    Let $\sigma \in \Delta^n_m$ and $x\in (M^A)_m$.
    For any $\theta: [k] \to [m]$ in $\Delta$,
    \begin{align*}
        (J_{A/R})_k\theta^*_N\big( f(\sigma \otimes x) \big) &= (J_{A/R})_kf\big( (\sigma \theta)\otimes \theta^*_M(x) \big) \\
        &= f\big( (\sigma \theta)\otimes (J_{A/R})_k\theta^*_M(x) \big) =0
    \end{align*}
    which implies $f(\sigma\otimes x)\in (N^A)_m$ and the claim holds.
    Clearly $Id_M$ induces $Id_{M^A}$ by restriction.
    Moreover, since restriction 
    respects morphism composition $\circ$ in Definition \ref{yydef3.4},
    $(-)^A$ is an enriched functor $A^e \sMod \to R\sMod$.
    
    (3): Define $K: \shom_{A^e}(A,M) \to M$ by $f \mapsto K(f)=f(Id_{[n]} \otimes 1_{A_n})$
    which is a simplicial map since
    \begin{align*}
        K(\theta^*(f)) &=\theta^*(f)(Id_{[k]}\otimes 1_{A_k}) =f(\theta \otimes 1_{A_k}) \\
        &=\theta^*\big( f(Id_{[n]}\otimes 1_{A_n})\big) = \theta^*(K(f))
    \end{align*}
    for any $\theta:[k]\to [n]$ in $\Delta.$
    Furthermore, $K$ is also an $R$-smodule morphism since \begin{align*}
        K(r\cdot f)=rf(Id_{[n]}\otimes 1_{A_n})=rK(f)
    \end{align*} for every $r\in R_n$ and $f\in \shom_{A^e}(A,M)_n.$
    For every $\theta:[k]\to [n]$ in $\Delta,$
    \begin{align*}
    (J_{A/R})_k \theta^*\big( K(f) \big) = (J_{A/R})_k f(\theta \otimes 1_{A_k})=f(\theta \otimes (J_{A/R})_k 1_{A_k})=0
    \end{align*}
    which implies that $K(f)\in (M^A)_n$.
    Therefore, $K$ determines an $R$-smodule morphism $\shom_{A^e}(A,M)\to M^A.$
    
    Conversely, define $\Omega:M^A  \to \shom_{A^e}(A,M)$ by
    sending every $x\in (M^A)_n$
    to the left $A^e$-smodule morphism $\Omega_n(x)$ where
    \begin{align*}
        \Omega(x)(\sigma \otimes a)=(a\otimes 1_{A_k})\cdot_M \sigma^*_{M}(x)=(1_{A_k}\otimes a)\cdot_M \sigma^*_M(x)
    \end{align*}
    for every $\sigma \in \Delta^n_k$ and $a\in A_k.$
    Then $A^e$-linearity of $\Omega(x)$ follows from $x\in (M^A)_n$.
    It is easy to verify that $\Omega$ is an $R$-smodule morphism and is the inverse of $K$.
    Therefore, both $K$ and $\Omega$ are $R$-smodule isomorphisms.

    It remains to show that $K$ is an $R\sMod$-enriched natural transformation.
    Let $\phi\in \shom_{A^e}(M,N)_n$ and consider the diagram 
    \[\begin{tikzcd}[column sep= large]
        \shom_{A^e}(M,N)_n \arrow[d, "\shom_{A^e}(A{,}-)"'] \arrow[r, "(-)^A"] & \shom_R(M^A, N^A)_n \arrow[d,"^*K"]\\
          \shom_R\big( \shom_{A^e}(A, M),\shom_{A^e}(A, N) \big)_n  \arrow[r,"K^*"] & \shom_R\big(\shom_{A^e}(A, M), N^A \big)_n
    \end{tikzcd}\]
    where $K^*=K \circ -$ and $^*K=-\circ (Id_{R\Delta^n}\sotimes_R K) $ as compositions of $0$-morphisms.
    Note that, for any $\sigma\in \Delta^n_k$ and $f\in  \shom_{A^e}(A,M)_k$, 
    \begin{align*}
        \big( \phi \circ (Id_{R\Delta^n} \sotimes_R K)\big)(\sigma \otimes f) &=\phi\big(\sigma \otimes K(f) \big)=\phi\big(\sigma \otimes f(Id_{[k]}\otimes 1_{A_k}) \big)\\
        &=\sigma^*(\phi)\big( Id_{[k]} \otimes f(Id_{[k]}\otimes 1_{A_k})\big)\\
        &=K(\sigma^*(\phi) \circ f).
    \end{align*}
    So the diagram commutes for every $n\geq 0$ and $R\sMod$-enriched naturality follows. Therefore, $K$ is an $R\sMod$-enriched natural isomorphism.
\end{proof}

The following proposition is the simplicial analogue of \cite[Corollary 4.1.7]{Ford2017}.
The exactness in the proposition below refers to the functor between the (underlying) abelian categories $A^e\sMod$ and $R\mathrm{sMod}$.

\begin{proposition}
\label{yypro3.10}
    Let $A$ be an $R$-salgebra.
    Then $A$ is separable over $R$ if and only if the centralizer functor $(-)^A: A^e\sMod \to R\sMod$ is exact.
\end{proposition}

\begin{proof}
    ($\implies$)
    Suppose $A$ is separable over $R.$
    Then there is an idempotent $e\in (A^e)_0$ such that $\mu(e)=1_{A_0}$ and $J_{A/R}e=0.$
    For any left $A^e$-smodule $M$,
    we claim that
    \begin{align}
    \label{E3.10.1}
    \tag{E3.10.1}
        M^A= eM \text{ as $R$-smodules.}
    \end{align}
    
    It is clear that $eM\subseteq M^A$ since $J_{A/R}(eM)=(J_{A/R}e)M=0$.
    Conversely, if $x\in (M^A)_k$,
    then $(J_{A/R})_k x=0$ which implies 
    $\big(1_{(A^e)_k}-s_0^k(e)\big)x=0$ 
    by Proposition \ref{yypro2.7}.
    This implies $x=s_0^k(e)x$.
    Therefore, $M^A \subseteq eM$ and the claim is proved.
    Since $e$ is an idempotent, $M$ decomposes as $R$-smodules as follows
    \begin{align*}
     M= eM \oplus (1_{(A^e)_0}-e)M = M^A \oplus  (1_{(A^e)_0}-e)M.
    \end{align*}
    Thus, $(-)^A$ is exact as it is the restriction to the direct summand $eM$.

    ($\impliedby$) Suppose $(-)^A$ is exact.
    Applying $(-)^A$ to the exact sequence of left $A^e$-smodules $0\to J_{A/R} \to A^e \xrightarrow[]{\mu_{A/R}} A \to 0,$
    we get a surjective $R$-smodule morphism $
    \overline{\mu_{A/R}}: (A^e)^A \to A^A.$
    Let $e \in ((A^e)^A)_0$ be such that $\overline{\mu_{A/R}}(e)=1_{A_0}$.
    By Definition \ref{yydef3.7}, $e \in ((A^e)^A)_0$ implies
    $(J_{A/R})_k s_0^k(e)=0$ for every $k\geq 0.$
    Therefore, $(J_{A/R})e=0$.
    Then $e$ is a separability idempotent and
    $A$ is separable by Proposition \ref{yypro2.7}.
\end{proof}

\begin{corollary}
\label{yycor3.11}
    Let $R$ be a commutative sring and $A$ be an $R$-salgebra.
    Then $\shom_{A^e}(A,-)$ is exact if and only if $\Hom_{A^e \sMod}(A,-)$ is exact.
\end{corollary}

\begin{proof}
    By Lemma \ref{yylem3.9}(3), Proposition \ref{yypro3.10}, and Proposition \ref{yypro2.7},
    \begin{align*}
        \text{$\shom_{A^e}(A,-)$ is exact} & \iff \text{$(-)^A$ is exact} \\
        & \iff \text{$A$ is separable over $R$} \\
        & \iff \text{$A$ is a projective left $A^e$-smodule}
        \\
        &\iff \text{$\Hom_{A^e \sMod}(A,-)$ is exact}. \qedhere
    \end{align*} 
\end{proof}

\begin{corollary}
\label{yycor3.13}
    Let $R$ be a commutative sring and $A$ be an $R$-salgebra.
    Then there is an $R\sMod$-enriched adjunction
    \begin{align}
    \label{E3.12.1}
    \tag{E3.12.1}
        A \sotimes_R -: R\sMod \leftrightarrows A^e\sMod : (-)^A.
    \end{align}
\end{corollary}

\begin{proof}
    By Theorem \ref{yythm3.6} and Lemma \ref{yylem3.9}(3), 
    there are $R$-smodule isomorphisms
    \begin{align*}
        \shom_{A^e}(A \sotimes_R V, M) \cong \shom_R(V,\shom_{A^e}(A,M)) \cong \shom_R(V,M^A)
    \end{align*}
    natural for any $R$-smodule $V$ and left $A^e$-smodule $M$.
\end{proof}

\subsection{Separable enriched functors} 
\label{yysec3.3}

It is well-known that separability of a monoid object or a monad is related to the separability of some functors (\cite[2.9(1)]{BBW2009}, \cite[Proposition 6.3] {BV2007}, and \cite[Proposition 1.3(1)]{NdBvO1989}).
We now formulate an enriched version of \cite[Sect. 1]{NdBvO1989}.

\begin{definition}
\label{yydef3.13}
    Let $\mathcal{C}, \mathcal{D}$ be $R \sMod $-enriched categories.
    We say an $R \sMod $-enriched functor $F: \mathcal{C} \to \mathcal{D}$ is {\it separable}
    if 
    for every pair of  objects $X, Y$ of $ \mathcal{C}$ and the associated $R$-smodule morphism $$F=F_{X,Y}: \shom_{\mathcal{C}}(X,Y) \to \shom_{\mathcal{D}}\big(F(X), F(Y)\big)$$
    there is an $R$-smodule morphism $H=H_{X,Y}: \shom_{\mathcal{D}}\big( F(X), F(Y)\big) \to \shom_{\mathcal{C}} ( X, Y )$ such that
    \begin{enumerate}
        \item[(1)] For every $X, Y\in \text{Obj } \mathcal{C} $, $$H\big( F(f) \big)= f, \forall f\in \shom_{\mathcal{C}}(X,Y).$$

        \item[(2)] 
        For any objects $M,N, M', N'$ in 
        $ \mathcal{C} $, 
        and any morphisms $f\in \shom_{\mathcal{C}}(N,N')_n$, $f'\in \shom_{\mathcal{C}}(M,M')_n$,
        $g\in \shom_{\mathcal{D}}\big(F(M), F(N) \big)_n$, and $g'\in \shom_{\mathcal{D}}\big( F(M'), F(N')\big)_n$
        $$F(f)\circ g = g' \circ F(f') \quad \implies \quad f \circ H(g) = H(g')\circ f'.$$
 \end{enumerate} 
\end{definition}

\begin{remark}
\label{yyrem3.14}
    Retain the notation in Definition \ref{yydef3.13}.
    We note
    \begin{enumerate}
        \item[(1)]  For any object $X$ in $\mathcal{C}$, $H(Id_{F(X)})=H\big( F(Id_X)\big)=Id_{X}$.

        \item[(2)] If $f'=Id_M$ and $g'=F(f)\circ g$, condition (2) in Definition \ref{yydef3.13} implies
        $f\circ H(g)= H\big( F(f)\circ g\big)$ (note we do not require $H$ form a functor).
    \end{enumerate}
\end{remark}

\begin{lemma}
\label{yylem3.15}
    Let $A$ be an $R$-salgebra.
    For any $M,N\in A\sMod,$ $\shom_R(M,N)^A = \shom_A(M,N)$
    as $R$-smodules.
\end{lemma}
\begin{proof}
    Note that $\shom_R(M,N)$ is naturally an $(A,A)$-sbimodule by viewing both $M$ and $N$ as $(A,R)$-sbimodules.
    Thus, it is easy to verify that $\shom_R(M,N)$ is a left $A^e$-smodule with action given by
    \begin{align}
    \label{E3.15.1}
    \tag{E3.15.1}
       \big( (a\otimes b)\cdot f \big)(\sigma \otimes x) = \sigma^*(a)f\big( \sigma \otimes \sigma^*(b)x \big)
    \end{align}
    for any $a,b\in A_m$, $f\in \shom_R(M,N)_m$, $\sigma \in \Delta^m_k$, and $x\in M_k$.

    Given any $f\in (\shom_R(M,N)^A)_n$,
    for any $\sigma \in \Delta^n_k$, $x\in M_k$, and $a\in A_k,$
    \begin{align*}
        f(\sigma \otimes ax)=\sigma^*(f) (Id_{[k]}\otimes ax)&= ((1_{A_k}\otimes a)\cdot \sigma^*(f))(Id_{[k]}\otimes x) &&\text{by \eqref{E3.15.1}}\\
        &=((a\otimes 1_{A_k}) \cdot \sigma^*(f))(Id_{[k]}\otimes x) &&\text{as $f\in (\shom_R(M,N)^A)_n$}\\
        &=a\sigma^*(f)(Id_{[k]}\otimes x)= af(\sigma\otimes x)
    \end{align*}
    which implies $f$ respects left $A$-action and therefore lies in $\shom_A(M,N)_n$.
    
    Conversely,
    given any $g\in \shom_A(M,N)_n$,
    for any $\theta:[m]\to[n]$ in $\Delta$ and $a\in A_m$,
    \begin{align*}
        ((a\otimes 1_{A_m}) \cdot \theta^*(g))(\sigma \otimes x) &= \sigma^*(a) \theta^*(g)(\sigma \otimes x) &&\text{by \eqref{E3.15.1}} \\
        &=\sigma^*(a) g(\theta \sigma\otimes x) &&\text{by \eqref{E3.1.1}}\\
        &=g(\theta\sigma \otimes \sigma^*(a)x) = \theta^*(g)(\sigma \otimes \sigma^*(a)x)\\
        &= ((1_{A_m}\otimes a)\cdot \theta^*(g))(\sigma \otimes x)
    \end{align*}
    for $\sigma\in \Delta^m_k$ and $x\in M_k.$
    This shows $(a\otimes 1_{A_m}) \cdot \theta^*(g)=(1_{A_m}\otimes a)\cdot \theta^*(g).$
    Therefore, $g\in (\shom_{R}(M,N)^A)_n.$
    This finishes the proof.
\end{proof}

The following extends \cite[Proposition 1.3(1)]{NdBvO1989} to $R\sMod$-enriched setting.

\begin{proposition}
\label{yypro3.16}
    Let $A$ be an $R$-salgebra.
    Then $A$ is separable over $R$ if and only if the forgetful functor $\mathcal{U}=\mathcal{U}_{A/R}:A \sMod \to R \sMod$ is a separable $R\sMod$-enriched functor.
\end{proposition}

\begin{proof}
    ($\implies$) Suppose $A$ is separable over $R.$ 
    Let $A^e= A \sotimes_R A^{o}.$
    Then there is a separability idempotent $e= \sum_{i=1}^t x_i \otimes y_i\in (A^e)_0$ satisfying Proposition \ref{yypro2.7}(3).
    By Lemma \ref{yylem3.15} and \eqref{E3.10.1},
    we obtain $R$-smodule equalities
    \begin{align}
    \label{E3.16.1}
    \tag{E3.16.1}
    \shom_A(M,N)=\shom_{R}(M,N)^A =e\shom_R(M,N).
    \end{align}
    Therefore, we may define a simplicial map \begin{align*}
        H: \shom_R(M,N) \to \shom_A(M,N) \ \ \  f \mapsto s_0^m(e)\cdot f
    \end{align*}
    for every $f\in \shom_R(M,N)_m$.
    Recall $\shom_R(M,N)$ and $\shom_A(M,N)$ are $R$-smodules as in Lemma \ref{yylem3.2}.
    Note the left $R$-actions on these two simplicial Hom-sets in \eqref{E3.2.1} coincide with the left action of $\eta_{A/R}(R)\otimes 1$ in \eqref{E3.15.1}.
    Since $e$ is a separability idempotent,
    it follows that $H$ is an $R$-smodule morphism 
    \begin{align*}
        H(r\cdot f)  &=  s_0^m(e)\cdot \big( (\eta_{A/R}(r) \otimes 1_{A_m}) \cdot f \big)  \\
        &= (\eta_{A/R}(r)\otimes 1_{A_m}) \cdot (s_0^m(e)\cdot f ) = r\cdot H(f).
    \end{align*}
    for every $r\in R_m$ and $f\in \shom_{R}(M,N)_m$.
    
    It remains to show $H$ satisfies conditions (1) and (2) of Definition \ref{yydef3.13}.
    
    We first show $H$ satisfies condition (1).
    Given any $h\in \shom_{A}(M,N)_m$,
    by \eqref{E3.16.1},
    write $h = s_0^m(e)\cdot \overline{h}$ for some $\overline{h}\in \shom_R(M,N)_m$.
    It follows that $$H(\mathcal{U}(h))=s_0^m(e)\cdot (s_0^m(e)\cdot \overline{h})= s_0^m(e)\cdot \overline{h} = h$$ as $s_0^m(e)$ is an idempotent.
    This shows condition (1) holds for $H.$

    Next we show $H$ satisfies condition (2).
    Suppose 
    $\mathcal{U}(f)\circ g = g'\circ \mathcal{U}(f')$ for left $A$-smodule $m$-morphisms $f, f'$ and $R$-smodule $m$-morphisms $g, g'$.
    Note that
    \begin{align*}
        (f\circ H(g))(\sigma \otimes x)
        &=f\big( \sigma \otimes (s_0^m(e)\cdot g)(\sigma \otimes x) \big) && \text{by \eqref{E3.4.3}} \\
        &=\sum_{i=1}^t f\big( \sigma \otimes s_0^k(x_i)g(\sigma \otimes s_0^k(y_i)x)\big) && \text{by \eqref{E3.15.1}}\\
        &=\sum_{i=1}^t s_0^k(x_i)f\big( \sigma \otimes g(\sigma \otimes s_0^k(y_i)x)\big) && \text{$f$ is a left $A$-smodule morphism}\\
        &=\sum_{i=1}^t s_0^k(x_i)(\mathcal{U}(f)\circ g)(\sigma \otimes s_0^k(y_i)x)\big) && \text{by \eqref{E3.4.3}} \\
        &=\big( s_0^m(e)\cdot (\mathcal{U}(f)\circ g) \big)(\sigma \otimes x) && \text{by \eqref{E3.15.1}}
    \end{align*}
    for every $\sigma \in \Delta^m_k$ and $x\in M_k$.
    This shows $f\circ H(g)=s_0^m(e)\cdot (\mathcal{U}(f) \circ g)$.
    By a similar computation,
    $H(g')\circ f' = s_0^m(e)\cdot (g'\circ  \mathcal{U}(f')).$
    Since 
    $\mathcal{U}(f)\circ g = g'\circ \mathcal{U}(f')$,
    condition (2) now follows from
    \begin{align*}
        f\circ H(g)= s_0^m(e)\cdot (\mathcal{U}(f) \circ g)=s_0^m(e)\cdot (g' \circ \mathcal{U}(f'))=H(g')\circ f'.
    \end{align*}
    Therefore,
    $\mathcal{U}$ is a separable $R\sMod$-enriched functor.

    ($\impliedby$) Suppose $\mathcal{U}$ is separable.
    Let $A^e = A \sotimes_R A^{o}.$
    Since the $R$-action is central in $A$ and $A^e$,
    there is an $R$-smodule morphism \begin{align*}
        \jmath: A \to A^{e}  \qquad x\mapsto 1_{A_n} \otimes x
    \end{align*} for every $x\in A_n$,
    which implies
    there is a left $A$-smodule morphism $H(\jmath): A \to A^e$.
    
    We claim $H(\jmath)$ is an $(A,A)$-sbimodule morphism.
    It suffices to show $H(\jmath)$ is a right $A$-smodule morphism.
    For every $a\in A_n$, define left $A$-smodule morphisms
    \begin{align*}
        &r_a \in \shom_A(A,A)_n && \sigma \otimes x \mapsto x \cdot \sigma^*(a) \\
        &\rho_a\in \shom_A(A^e, A^e)_n && \sigma \otimes z \mapsto (1_{A_k}\otimes \sigma^*(a))\cdot z
    \end{align*}
    for every $\sigma \in \Delta^n_k$, $x\in A_k$, and $z\in (A^e)_k$.
    Note
    $r_a$ represents the right multiplication by $a$ on $A$, 
    while $\rho_a$ represents the right multiplication by $a$ on $A^e.$
    
    For any $a\in A_n$, we claim that
    \begin{align}
    \label{E3.16.2}
    \tag{E3.16.2}
        \mathcal{U}(\rho_a) \circ s_0^n(\jmath) = s_0^n(\jmath) \circ \mathcal{U}(r_a).
    \end{align} 
    Indeed,
    for every $\sigma \in \Delta^n_k$ and $x\in A_k$,
    \begin{align*}
        \big(\mathcal{U}(\rho_a) \circ s_0^n(\jmath) \big)(\sigma \otimes x) &= \rho_a(\sigma \otimes 1_{A_k}\otimes x) 
        =(1_{A_k}\otimes \sigma^*(a))(1_{A_k}\otimes x)\\
        &=1_{A_k}\otimes x\sigma^*(a)
        =1_{A_k} \otimes r_a(\sigma \otimes x)
        =\big( s_0^n(\jmath) \circ \mathcal{U}(r_a) \big)(\sigma \otimes x).
    \end{align*}
    This proves \eqref{E3.16.2}.
    By condition (2) of $H$,
    \eqref{E3.16.2} implies
    \begin{align}
    \label{E3.16.3}
    \tag{E3.16.3}
        \rho_a \circ H(s_0^n(\jmath)) = H(s_0^n(\jmath)) \circ r_a.
    \end{align}
    Evaluating \eqref{E3.16.3} at $Id_{[n]} \otimes x$ for any $x\in A_n$ (also note $H(s_0^n(\jmath))=s_0^n(H(\jmath))$),
    \begin{align*}
        (1_{A_n}\otimes a)\cdot H(\jmath)(x)=H(\jmath)(xa).
    \end{align*}
    Therefore, $H(\jmath)$ is a right $A$-smodule morphism.
    It follows that $H(\jmath)$ is an $(A,A)$-sbimodule morphism, and equivalently, a left $A^e$-smodule morphism.
    
    Last, we note $H(\jmath)$ splits $\mu_{A/R}: A^e \to A$ since, by Remark \ref{yyrem3.14}(1,2),
    \begin{align*}
        \mu_{A/R} \circ H(\jmath) = H\big( \mathcal{U}(\mu)\circ \jmath \big) = H(Id_A)= Id_{A}.
    \end{align*}
    Then $A$ is separable over $R$ by Proposition \ref{yypro2.7}(2).
\end{proof}

\section{More Properties on Separability}
\label{yysec4}

We collect basic properties of separable simplicial algebras,
many of which have apparent analogues in the classical theory of separable algebras \cite[Chapter 4]{Ford2017}.

\subsection{Invariance under base change, tensor product, and quotient}
\label{yysec4.1}

\begin{lemma}
\label{yylem4.1}
    Let $R$ be a commutative sring and $S$ be a commutative $R$-salgebra.
    If $A$ is a separable $R$-salgebra,
    then $S \sotimes_R A$ is a separable $S$-salgebra.
\end{lemma}

\begin{proof}
    Let $T= S \sotimes_R A$. 
    Note there is a canonical sring isomorphism
    \begin{align*}
        T^e = (S \sotimes_R A)\sotimes_S (S \sotimes_R A)^{o}  \cong S\sotimes_R A^e 
    \end{align*} 
    where for every $s_0, s_1\in S_n$ and $a_0, a_1\in A_n$, 
    \begin{align*}
        s_0 \otimes_R a_0 \otimes_S s_1 \otimes_R a_1 \mapsto s_0s_1 \otimes_R a_0 \otimes_R a_1.
    \end{align*}
    Therefore,
    the left $T^e$-action on $T$
    is identified with
    a left $S \sotimes_R A^e$-action on $T$ 
    \begin{align*}
        (s\otimes_R a_0 \otimes_R a_1)\cdot (x \otimes_R a) = sx \otimes_R a_0aa_1
    \end{align*}
    for $s, x\in S_n$ and $a, a_0, a_1\in A_n$ and the left $T^e$-smodule morphism $\mu_{T/S} : T^e \to T$
    is identified 
    as left $S \sotimes_R A^e$-smodule morphism $\tilde{\mu}: S\sotimes_R A^e \to T$ given by
    \begin{align*}
        s \otimes a_0 \otimes a_1 \mapsto s \otimes a_0a_1.
    \end{align*}
    It follows that $\tilde{\mu}= Id_S \sotimes_R \mu_{A/R}.$
    Since $A$ is separable over $R,$
    $\mu_{A/R}$ admits  a left $A^e$-smodule section $h.$
    Then the left $S \sotimes_R A^e$-smodule morphism $Id_S \sotimes_R h$ splits $\tilde{\mu}$.
    Therefore, $\mu_{T/S}$ splits.
    By Proposition \ref{yypro2.7}(2),
    $T$ is separable over $S.$ 
\end{proof}

\begin{lemma}
\label{yylem4.2}
Let $R$ be a commutative sring 
and $S, T$ be commutative 
$R$-salgebras.
Suppose $A$ is a separable 
$S$-salgebra and $B$ is a 
separable $T$-salgebra.
Then $A\sotimes_R B$ is separable over $S\sotimes_R T.$
\end{lemma}

\begin{proof}
Let $A^e=A \sotimes_S A^{o}$ 
and $B^e=B \sotimes_T B^{o}$.
Let $D= A \sotimes_R B$ and 
$K=S \sotimes_R T.$
Let $D^e= D \sotimes_K D^{o}.$
There is a canonical sring 
isomorphism
\begin{align*}
D^e = (A\sotimes_{R} B )\sotimes_{K} (A\sotimes_{R} B)^{o}  
\cong A^e \sotimes_R B^e
\end{align*}
where for every $a_0, a_1\in A_n$ and $b_0, b_1\in B_n$,
    \begin{align*}
        a_0 \otimes_R b_0 \otimes_{K} a_1 \otimes_R b_1 \mapsto a_0\otimes_S a_1 \otimes_R b_0 \otimes_T b_1
    \end{align*}
    The left $D^e$-action on $D$ can be identified with the following left $A^e\sotimes_R B^e$-action
    \begin{align*}
        (a_0\otimes_S a_1 \otimes_R b_0 \otimes_T b_1)\cdot (a\otimes_R b) = a_0aa_1\otimes_R b_0bb_1
    \end{align*}
    for any $a,a_0, a_1\in A_n$ and $b, b_0, b_1 \in B_n.$
    Therefore,
    the left $D^e$-smodule morphism $\mu_{D/K}: D^e \to D$  can be identified with $\tilde{\mu}: A^e \sotimes_R B^e \to A \sotimes_R B$ by
    \begin{align*}
       a_0\otimes_S a_1 \otimes_R b_0 \otimes_T b_1 \mapsto a_0a_1 \otimes_R b_0b_1.
    \end{align*}
    Clearly $\tilde{\mu}= \mu_{A/S} \sotimes_R \mu_{B/T}$.
    Since $A$ is separable over $S$ and $B$ is separable over $T$,
    by Proposition \ref{yypro2.7}(2),
    $\mu_{A/S}$ admits a left $A^e$-smodule section $h_A$ and  $\mu_{B/T}$ admits a left $B^e$-smodule section $h_B$.
    So $h_A \sotimes_R h_B$ is a left $A^e\sotimes_R B^e$-smodule section of $\tilde{\mu}.$
    Hence $\mu_{D/K}$ also splits, 
    and 
    by Proposition \ref{yypro2.7},
    $D$ is separable over $K.$ 
\end{proof}

\begin{remark}
\label{yyrem4.3}
        Let $A$ be a separable $R$-salgebra and $e = \sum_{i=1}^t x_i \otimes y_i$  be a separability idempotent.
        Then
        \begin{enumerate}
            
            \item[(1)] $A^o$ is separable over $R$ because $e = \sum_{i=1}^t y_i \otimes x_i$ is a separability idempotent.
            Consequently, $A$ is separable over $R$ if and only if $A^{o}$ is separable over $R.$

            \item[(2)] 
            Then by Lemma \ref{yylem4.2},
            $A^e=A \sotimes_R A^{o}$ is separable over $R.$
        \end{enumerate}
    \end{remark}

\begin{lemma}
\label{yylem4.4}
    Let $R$ be a commutative sring and $A$ be a separable $R$-salgebra.
    Suppose $I$ is an sideal of $A$.
    Then $A/I$ is separable over $R$.
\end{lemma}

\begin{proof}
    Let $M$ be a left $(A/I) \otimes_R (A/I)^{o}$-smodule.
    By pulling back along the quotient map $A \twoheadrightarrow A/I$,
    $M$ is a left $A \otimes_R A^{o}$-smodule.
    Since the $A/I$-action on $M$ agrees with the $A$-action on $M$,
    $M^A = M^{A/I}$ as $R$-smodules.
    Since $A$ is separable over $R$,
    by Proposition \ref{yypro3.10},
    $(-)^A$ is an exact functor.
    Therefore, the functor $(-)^{A/I}$ is also exact which implies $A/I$ is separable over $R$
    by Proposition \ref{yypro3.10}.\qedhere
\end{proof}

\subsection{Simplicial center is computed levelwise}
\label{yysec4.2}
Let $R$ be a commutative sring.

\begin{lemma}
\label{yylem4.5}
Let $A$ be a separable $R$-salgebra. 
Then there is an $R$-smodule map $P:A \to A$ such that for every $n\geq 0$,
$${\rm im}(P_n) = Z(A_n) \text{ and } P_n(a)= a$$ for all $a\in Z(A_n).$
\end{lemma}

\begin{proof}
    Let $e= \sum_{i=1}^t {x_i\otimes y_i} \in (A^e)_0$ be a separability idempotent such that $\mu_{A/R}(e)=1_{A_0}$ and
    $(b\otimes 1_{A_k})s_0^k(e) = (1_{A_k} \otimes b)s_0^k(e)$
    for every $b\in A_k.$
    
    Because $A$ is a left $A^e$-smodule,
    there exists an $R$-salgebra morphism $$\varphi: A^e \to \shom_R(A,A)$$ given by the left $A^e$-action.
    Define
    $P:=\varphi(e)$ which maps $ a \mapsto s_0^n(e)\cdot a$ for all $a\in A_n.$
    Then by \eqref{E3.10.1} and Remark \ref{yyrem3.8}(2),
    $${\rm im}(P) = eA = A^A= Z(A).$$
    Moreover, if $a\in Z(A_n)$,
    then $$
        a = \mu_{A/R}(s_0^n(e))a= \sum_{i=1}^t s_0^n(x_i) a s_0^n(y_i) = s_0^n(e) \cdot a = P(a). $$
        This completes the proof.
\end{proof}

\begin{proposition}
[Level-wise center lemma]
\label{yypro4.6}
Let $A$ be a separable $R$-salgebra. 
Then for every $n\geq 0$,
$$Z(A)_n=Z(A_n).$$
\end{proposition}

\begin{proof}
    By definition,
    $Z(A)_n\subseteq Z(A_n)$.
    The converse inclusion also holds because,
    for any $a\in Z(A_n)$, by Lemma \ref{yylem4.5},
    $a = P(a) \in {\rm im}(P)= Z(A)_n.$
\end{proof}


\begin{corollary}
\label{yycor4.7}
    Let $A$ and $B$ be separable $R$-salgebras.
    Then $$Z(A\sotimes_R B) = Z(A)\sotimes_R Z(B).$$
\end{corollary}

\begin{proof}
    By Lemma \ref{yylem4.2},
    $A\sotimes_R B$ is separable over $R$.
    In particular,
    for every $n\geq 0$,
    $A_n\otimes_{R_n} B_n$ is separable over $R_n$.
    By \cite[Chapter 3, Corollary 2.7]{Bass1967},
    $$Z(A_n\otimes_{R_n}B_n)=Z(A_n) \otimes_{R_n}Z(B_n).$$
    It follows from Proposition \ref{yypro4.6} that for every $n\geq 0$
    $$Z(A\sotimes_R B)_n = Z(A_n \otimes_{R_n} B_n)= Z(A_n) \otimes_{R_n}Z(B_n)= Z(A)_n\otimes_{R_n} Z(B)_n.$$
    This completes the proof.
\end{proof}

\subsection{Contrasts with the classical theory}
\label{yysec4.3}

Section \ref{yysec4.1} demonstrates a large portion of the classical theory of separable algebras is retained in the simplicial setting.
There are, nevertheless, important differences.
In particular, 
\begin{enumerate}
\item 
Classically, simple commutative rings are necessarily fields.
This observation is used in Van Oystaeyen and Y. Zhang's proof of \cite[Proposition 4.6]{VOZ1998}.
    But simple commutative srings may not be simplicial fields.

    \item Simplicial separability implies levelwise separability but not conversely.
    This shows simplicial separability is strictly stronger than levelwise separability. 

\end{enumerate}

Let $\Bbbk$ be a commutative ring.
Define $W^{\langle 0 \rangle}:=W^{\langle 0 \rangle}/\Bbbk$ to be the  $\Bbbk$-smodule corresponding, under the Dold-Kan correspondence (\cite[Theorem 8.4.1]{Weibel1994} and \cite[Theorem 3.1]{HSZZ26}), to the chain complex $ \Bbbk \cong \Bbbk $ concentrated in degrees $1$ and $0.$
We may denote the $\Bbbk$-linear basis of $(W^{\langle 0 \rangle})_n$ by $\{\xi^n_j\}_{j=0}^n$.
Explicitly, the face maps are given by \begin{align*}
    d_i(\xi^n_j) = \begin{cases}
        \xi^{n-1}_j & i \geq j+1 \\
        \xi^{n-1}_{j-1} & i \leq j
    \end{cases}
\end{align*}
with the convention that $\xi^n_{-1}=0$, and the degeneracy maps are given by
\begin{align*}
    s_i(\xi^n_j) = \begin{cases}
        \xi^{n+1}_j & i\geq j+1 \\
        \xi^{n+1}_{j+1} & i \leq j
    \end{cases}.
\end{align*}
Moreover, there is a unital commutative $\Bbbk$-salgebra structure on $W^{\langle 0\rangle}$:
\begin{align*}
    \xi^n_i\xi_j = \xi^n_j \xi^n_i = \xi^n_{\min\{i,j\}}
\end{align*}
where $\xi^n_n=1_{(W^{\langle 0\rangle})_n}$ plays the role of unit in $(W^{\langle 0\rangle})_n.$
We note that 
$$p^n_j = \xi^n_j-\xi^n_{j-1} \quad p^n_0= \xi^n_0$$ 
form a complete set of orthogonal idempotents of $(W^{\langle 0 \rangle})_n$.
So $(W^{\langle 0 \rangle})_n \cong \Bbbk^{\oplus (n+1)}$ is not a simple $\Bbbk$-algebra for $n>0$.
We call $W^{\langle 0 \rangle}/\Bbbk$ the {\it simplicial winding algebra} (with base ring $\Bbbk$) \cite[Sect. 6.2]{HSZZ26}.

\begin{lemma}
\label{yylem4.8}
Let $\Bbbk$ be a field.
Then $W^{\langle 0 \rangle}/\Bbbk$ is simple and not a simplicial field.
\end{lemma}

\begin{proof}
    A simplicial field is necessarily constant because the face maps are surjective by the simplicial identities.
    But any nonzero surjection between fields is an isomorphism.
    Since $W^{\langle 0 \rangle}$ is not a constant sring,
    it is not a simplicial field.

    Suppose $I$ is a nonzero proper sideal of $W^{\langle 0 \rangle}/\Bbbk.$
    Then the normalized complex $N(I)$ of $I$ must be a nonzero subcomplex of $N(W^{\langle 0 \rangle}/\Bbbk) = (\Bbbk \cong \Bbbk).$
    In particular,
    $I_0=N(I)_0 \neq 0$ for otherwise $N(I)_1=0$ as the differential map $N(I)_1\to N(I)_0$ is injective.
    Then $I_0$ is a nontrivial ideal of $(W^{\langle 0 \rangle}/\Bbbk)_0=\Bbbk$,
    which means $I_0=\Bbbk.$
    It follows that $1_{(W^{\langle 0 \rangle}/\Bbbk)_0}\in I_0$ and $1_{(W^{\langle 0 \rangle}/\Bbbk)_n} \in I_n$ for every $n\geq 0$.
    Therefore $I = W^{\langle 0 \rangle}/\Bbbk,$ contradicting $I$ being a proper sideal.
\end{proof}

\begin{example}
\label{yyexa4.9}
    Let $K/\Bbbk$ be a finite separable field extension.
    Then $$(W^{\langle 0 \rangle}/\Bbbk)\sotimes_{\Bbbk} K \cong W^{\langle 0 \rangle}/K$$  is a simple commutative $K$-salgebra.
    Recall from Theorem \ref{yythm2.11} that
    any separable $K$-salgebra $A$ is of the form $ K\Delta^0\sotimes_{K} A_0$.
    \begin{enumerate}
        \item[(1)] As $W^{\langle 0 \rangle}/K$ is not a constant sring,
        $W^{\langle 0 \rangle}/K$ is not separable over $K,$
        even though each $(W^{\langle 0 \rangle}/K)_n = K^{\oplus (n+1)}$ is separable 
        over $K$.
        Therefore,
        levelwise separability is weaker than simplicial separability by Remark \ref{yyrem2.9}(3).

    \item[(2)]  Note that $K$ is separable over $\Bbbk$ \cite[Theorem 4.5.6]{Ford2017}.
    So $W^{\langle 0 \rangle}/K$ is separable over $W^{\langle 0 \rangle}/\Bbbk$ by Lemma \ref{yylem4.1}.
    This gives yet another example, 
    showing that over a non-constant commutative sring, 
    a separable simplicial algebra need not be constant, 
    in contrast to Theorem \ref{yythm2.11}.
    \end{enumerate}
\end{example}

\section{Enriched Morita Theory}
\label{yysec5}
The enriched Morita theory of monoids in a monoidal category is well-developed from a categorical perspective \cite{ BV2002,LP2024,PareigisIII}.
We record some of the general theory of enriched Morita equivalence and characterize the relevant notions of finitely generated projective modules and generators for simplicial modules as finite retracts.

\subsection{Simplicial Morita context}
\label{yysec5.1}

Let $R$ be a commutative simplicial ring.
A {\it Morita context} is a sextuple $(A,B, _AP_B, _BQ_A, f, g)$ consisting of $R$-salgebras $A, B$,
an $(A,B)$-sbimodule $P$, 
and a $(B,A)$-sbimodule $Q$
 with
\begin{align*}
    &f: Q \sotimes_A P \to B \quad \text{a $(B,B)$-sbimodule morphism}\\
    &g: P \sotimes_B Q \to A \quad \text{an $(A,A)$-sbimodule morphism}
\end{align*}
making the following diagrams commutative:
\[
\begin{tikzcd}
    Q \sotimes_A P \sotimes_B Q \arrow[r, "f \sotimes_B Id"] \arrow[d, "Id \sotimes_A g"] & B \sotimes_B Q \arrow[d, "\cong"]\\
    Q \sotimes_A A \arrow[r,"\cong"] & Q
\end{tikzcd}
\qquad
\begin{tikzcd}
P \sotimes_B Q \sotimes_A P \arrow[r, "g \sotimes_A Id"] \arrow[d, "Id \sotimes_B f"] & A \sotimes_A P \arrow[d, "\cong"]\\
    P \sotimes_B B \arrow[r,"\cong"] & P
\end{tikzcd}
\]
where $\cong$ are the natural isomorphisms.

\begin{definition}
\label{yydef5.1}
    Let $P$ be a right $B$-smodule.
    Define $S_{P}: = \shom_B(P, P)$.
    Then $P$ is an $(S_P,B)$-sbimodule under actions given by
    \begin{align*}
        &P \sotimes_R B \to P  && p\otimes b \mapsto pb \\
        &S_{P} \sotimes_R P \to P  &&s\otimes p \mapsto sp:=s(Id_{[n]}\otimes p).
    \end{align*}
    Consequently, $Q_{P}:= \shom_B(P,B)$ is a $(B,S_{P} )$-sbimodule with actions given by
    \begin{align*}
        & Q_{P} \sotimes_R S_{P}  \to Q_{P} && q \otimes s \mapsto \big( qs: \sigma \otimes p\mapsto q(\sigma \otimes s(\sigma \otimes p)) \big)\\
        & B \sotimes_R Q_{P} \to Q_{P} && b \otimes q \mapsto \big( bq : \sigma \otimes p \mapsto \sigma^*(b)q(\sigma\otimes p) \big).
    \end{align*}
    Define a $(B,B)$-sbimodule morphism and an $(S_P,S_P)$-sbimodule morphism
    \begin{align*}
        &\mathfrak{a}: Q_{P} \sotimes_{S_{P}} P \to B \qquad q \otimes_{S_P} p \mapsto \big( qp=q(Id_{[n]} \otimes p) \big) \\
        &\mathfrak{b}: P \sotimes_B Q_{P} \to S_{P}  \qquad p \otimes_B q \mapsto \big( pq: \sigma' \otimes p' \mapsto (\sigma')^*(p)  q(\sigma' \otimes p')\big).
    \end{align*}
    We say $( S_{P} , B, {}_{S_{P}}P_B,{}_B(Q_{P}){}_{S_{P}}, \mathfrak{a}, \mathfrak{b})$ is the {\it canonical Morita context} associated with $P_B$ in $R\sMod.$
    Moreover,
    we say
    \begin{enumerate}
        \item $P_B$ is {\it $B$-basic projective} if $\mathfrak{b}$ is a surjective $(S_P,S_P)$-sbimodule morphism.

        \item $P_B$ is a {\it $B$-sgenerator} if $\mathfrak{a}$ is a surjective $(B,B)$-sbimodule morphism.

        \item $P_B$ is a {\it $B$-sprogenerator} if $P_B$ is a basic projective $B$-sgenerator.
        
    \end{enumerate}
\end{definition}

\begin{remark}
\label{yyrem5.2}
\begin{enumerate}
\item[(1)] 
For a left $A$-smodule $_AP$,
the canonical Morita context 
associated with $_AP$ together 
with $A$-basic projective/ $A$-sgenerator/ $A$-sprogenerator 
can be defined similarly.
\item[(2)] 
Definition \ref{yydef5.1} agrees with \cite[§2]{VOZ1998} and \cite{PareigisIII} specialized to simplicial modules.
\item[(3)] 
Definition \ref{yydef5.1} is also
the direct simplicial analogue of 
\cite[§18]{Lam1999}.
Moreover, it is not hard to verify 
that the simplicial analogues of 
the usual associativity identities 
in \cite[§18]{Lam1999} hold.
Consequently, if $\mathfrak{a} $ 
and $\mathfrak{b}$ are surjective, they are isomorphisms.
\end{enumerate}
\end{remark}

For simplicial modules,
basic projectives and sgenerators admit a simple characterization similar to that of finitely generated projectives and generators in the classical ring and module theory.

\begin{lemma}
\label{yylem5.3}
Let $P$ be a right $B$-smodule.
\begin{enumerate}
\item[(1)] 
$P_B$ is basic projective if and 
only if $P$ is a direct summand of 
some $B^{\oplus t}$.
\item[(2)] 
$P_B$ is an sgenerator if and only if $B$ is a direct summand of some $P^{\oplus t'}$.
    \end{enumerate}
\end{lemma}

\begin{proof}
Set $Q: =Q_P=\shom_B(P,B)$ and $S: = S_P = \shom_B(P,P)$.
Consider the canonical Morita context associated with $P_B$
$$( S, B, {}_{S}P_B,{}_BQ{}_{S}, \mathfrak{a}, \mathfrak{b}).$$

    (1)
    Suppose $P_B$ is basic projective.
    Since $\mathfrak{b}: P \sotimes_B Q \to S$ is surjective,
    there exist finitely many $p_i\in P_0$ and $q_i\in Q_0$ such that 
    $1_{S_0} = \sum_{i=1}^t p_iq_i.$
    Then,
    for every $p\in P_k$,
    we have
    \begin{align*}
        p &= 1_{S_k}\cdot p = s_0^k(\sum_{i=1}^t p_iq_i) p =\sum_{i=1}^t (s_0^k(p_i)s_0^k(q_i)) p  \\
        &= \sum_{i=1}^t s_0^k(p_i)(s_0^k(q_i)(Id_{[k]} \otimes p)) = \sum_{i=1}^t s_0^k(p_i)(q_i(p)).
    \end{align*}
    Define right $B$-smodule morphisms
    \begin{align*}
        &f: B^{\oplus t} \to P \qquad (b_1,\dots, b_t) \mapsto \sum_{i=1}^t s_0^k(p_i)b_i\\
        &h: P \to B^{\oplus t} \qquad p \mapsto (q_1(p), \dots, q_t(p)).
    \end{align*}
    It follows that $fh= Id_P$ which implies $P$ is a direct summand of $B^{\oplus t}.$

    Conversely,
    suppose $P$ is a direct summand of $B^{\oplus t}$.
    Then we have a retract
    \begin{align*}
        P \xrightarrow[]{h} B^{\oplus t} \xrightarrow[]{f} P \qquad fh=Id_P
    \end{align*}
    Let $\jmath_i: B \to B^{\oplus t}$ and $\pi_i: B^{\oplus t} \to B$ be the canonical inclusion and projection maps of the direct sum $B^{\oplus t}$.
    Define 
    \begin{align*}
        p_i = f\jmath_i(1_{B_0})\in P_0 \qquad q_i = \pi_ih \in Q_0.
    \end{align*}
    Then for every $p\in P_k$, since $f$ is a right $B$-smodule morphism
    \begin{align*}
        p &= fh(p) = f(\sum_{i=1}^t \jmath_i(1_{B_k})\pi_ih(p) ) = \sum_{i=1}^t f(\jmath_i(s_0^k(1_{B_0}))\pi_ih(p) ) \\
        &= \sum_{i=1}^t s_0^k(f\jmath_i(1_{B_0}) ) \pi_ih(p) =\sum_{i=1}^t s_0^k(p_i)q_i(p).
    \end{align*} 
    This implies that $1_{S_0} = \sum_{i=1}^t p_iq_i.$
    Therefore, for every $s\in S_k,$
    \begin{align*}
        s = s\cdot 1_{S_k}= s \sum_{i=1}^t s_0^k(p_i)s_0^k(q_i) =\sum_{i=1}^t \mathfrak{b}(ss_0^k(p_i)\otimes s_0^k(q_i))
    \end{align*}
    as $\mathfrak{b}$ is an $(S,S)$-sbimodule morphism.
    Therefore, $\mathfrak{b}$ is surjective.

    (2)
    Suppose $P_B$ is an sgenerator.
    Since $\mathfrak{a}: Q\sotimes_S P \to B$ is surjective,
    there exist finitely many $q_i'\in Q_0$ and $p_i'\in P_0$ such that $1_{B_0} = \sum_{i=1}^{t'} q_i'p_i'.$
    Define right $B$-smodule morphisms
    \begin{align*}
        &f': P^{\oplus t'} \to B \qquad (x_1,\cdots ,x_{t'}) \mapsto  \sum_{i=1}^{t'}  q_i'(x_i) \\
        &h': B \to P^{\oplus t'} \qquad y \mapsto (s_0^k(p_1')y,\dots, s_0^k(p_{t'}')y )
    \end{align*}
    Then by construction, $f'h'(1_{B_0})= 1_{B_0}$ which implies the right $B$-smodule morphism $f'h'$ is $Id_B.$
    Therefore, $B$ is a direct summand of $P^{\oplus t'}.$

    Conversely, suppose $B$ is a direct summand of $P^{\oplus t'}$. 
    So there is retract
    \begin{align*}
        B \xrightarrow{h'} P^{\oplus t'} \xrightarrow{f'} B \qquad f'h'=Id_B
    \end{align*}
    Write $h'(1_{B_0}) = (p_1', \cdots, p_{t'}')$
    and define
    \begin{align*}
        q_i' = P \xrightarrow{\jmath_i} P^{\oplus t'} \xrightarrow{f'} B
    \end{align*}
    Then $1_{B_0} = f'h'(1_{B_0})=\sum_{i=1}^{t'}  q_i'p_i'=\mathfrak{a}(\sum_{i=1}^{t'}  q_i'\otimes p_i').$
    Therefore, $\mathfrak{a}$ is surjective.
\end{proof}

\subsection{Simplicial Morita equivalence}
\label{yysec5.2}

\begin{theorem}
\label{yythm5.4}
\cite[Theorems 5.1 and 5.3]{PareigisIII}
Let $A$ and $B$ be $R$-salgebras.
The following are equivalent.
\begin{enumerate}
\item[(1)] 
There exists an $R\sMod$-enriched 
equivalence $F: A\sMod 
\leftrightarrows B \sMod: G$.
\item[(2)] 
There exists a Morita context 
$(A, B, {}_AP_B, {}_BQ_A, f, g)$ 
where $f$ and $g$ are surjective 
{\rm{(}}equivalently isomorphisms 
by Remark \ref{yyrem5.2}(3){\rm{)}}.
\end{enumerate}
In particular, 
there are $R\sMod$-enriched 
natural isomorphisms
\begin{align*}
F\cong Q \sotimes_A -  
\qquad G \cong P \sotimes_B  -
\end{align*}
and $_AP$ is an $A$-sprogenerator 
and $_BQ$ is a $B$-sprogenerator.
\end{theorem}

\begin{theorem}
\label{yythm5.5}
\cite[Theorem 2.1]{VOZ1998}
Let $(A,B, {}_AP_B, {}_BQ_A, f, g)$ 
be a Morita context in $R\sMod$.
Suppose $f,g$ are surjective {\rm{(}}equivalently isomorphisms{\rm{)}}.
Then 
\begin{enumerate}
\item[(1)] 
There are canonical $R$-salgebra isomorphisms induced by the sbimodule structures on $P$ and $Q$:
        \begin{align*}
           &A \cong \shom_B(P,P) && A^o \cong \shom_{B}(Q,Q)\\
           & B \cong \shom_A(Q,Q)  && B^o \cong 
        \shom_{A}(P,P). \end{align*}


        \item[(2)] The smodules ${}_AP,  P_B, {}_BQ, Q_A$ are sprogenerators.
    \end{enumerate}
\end{theorem}

\begin{corollary}
\label{yycor5.6}
    Let $A$ and $B$ be $R$-salgebras.
    Suppose $_AP_B$ is an $(A,B)$-sbimodule.
    The following are equivalent.
    \begin{enumerate}
        \item[(1)] $P_B$ is an sprogenerator and the canonical $R$-salgebra map $\lambda: A \to \shom_B(P,P)$ given by the left $A$-action is an  isomorphism.

        \item[(2)] $_AP$ is an sprogenerator and the canonical $R$-salgebra  map $\rho: B^o \to \shom_{A}(P,P)$ given by the right $B$-action is an  isomorphism.

    \end{enumerate}
   
\end{corollary}

\begin{proof}
    (1) $\implies$ (2) 
    Consider the canonical Morita context associated with $P_B$ 
    $$( S_P,B,  {}_{S_P}P_B,{}_B(Q_P){}_{S_P}, \mathfrak{a}, \mathfrak{b}).$$
    Since $P_B$ is a $B$-sprogenerator,
    both $\mathfrak{a}, \mathfrak{b}$ are isomorphisms.
    Given the isomorphism 
    $$\lambda: A \to S_P= \shom_B(P,P)$$
    the $S_P$-actions on $P$ and $Q_P$ can be identified with $A$-actions by pulling back along $\lambda$.
    In particular,
    the induced $A$-actions recover the original $(A,B)$-sbimodule structure on $P$ and $(B,A)$-sbimodule structure on $Q_P$ via 
    $$a\cdot p =\lambda(a) \cdot p .$$
    Therefore,
    there is a canonical $(B,B)$-sbimodule isomorphism induced by $\lambda$
    \begin{align*}
        \overline{\lambda}: Q_P \sotimes_A P \cong Q_P \sotimes_{S_P} P \qquad q\otimes_A p \mapsto q \otimes_{S_P} p.
    \end{align*}
    This induces a Morita context $(A, B, {}_AP_B,{}_B (Q_{P})_A, f,g)$ with isomorphisms
    \begin{align*}
        &f: Q_P \sotimes_A P \xrightarrow[]{\overline{\lambda}} Q_P \sotimes_{S_P} P \xrightarrow[]{\mathfrak{a}} B \\
        &g: P \sotimes_{B} Q_P \xrightarrow[]{\mathfrak{b}} S_P \xrightarrow[]{\lambda^{-1}} A.
    \end{align*}
    Indeed, $f$ and $g$ satisfy the compatibility expressed in the commutative diagrams before Definition \ref{yydef5.1},
    since (recall $a \cdot p'  = \lambda(a)\cdot p'$)
    \begin{align*}
        g(p\otimes_B q)\cdot p' &= \lambda (g(p\otimes_B q)) \cdot p'=\lambda(\lambda^{-1}\mathfrak{b}(p\otimes_B q)) \cdot p'\\
        &=\mathfrak{b}(p \otimes_{B} q)\cdot p'=p \cdot \mathfrak{a}(q 
        \otimes_{S_P} p') = p\cdot f(q \otimes_A p')
    \end{align*}
    and similarly 
    \begin{align*}
        f(q \otimes_A p)\cdot q' &=\mathfrak{a}(q \otimes_{S_P} p)\cdot q' =q \cdot \mathfrak{b}(p\otimes_B q')\\
        &=q \cdot \lambda(\lambda^{-1}\mathfrak{b}(p\otimes_B q')) = q \cdot g(p \otimes_B q').
    \end{align*}
    By Theorem \ref{yythm5.5},
    $_AP$ is an sprogenerator and the canonical map $B^o \to \shom_{A}(P,P)$ is an isomorphism.

    (2) $\implies $ (1) This is symmetric to the previous argument.
    We omit the proof.
\end{proof}

\begin{definition}
\label{yydef5.7}
An $(A,B)$-sbimodule $P$ is 
{\it invertible} if there is a 
$(B,A)$-sbimodule $Q$ where 
$P \sotimes_B Q \cong A$ as 
$(A,A)$-sbimodules and 
$Q \sotimes_A P \cong B$ as 
$(B,B)$-sbimodules.
\end{definition}

\begin{theorem}
\label{yythm5.8}
    Let $A$ and $B$ be $R$-salgebras.
    The following are equivalent.
\begin{enumerate}
\item[(1)] There exists an $R\sMod$-enriched equivalence $G: B\sMod  \to A\sMod.$
\item[(2)] There exists a $B$-sprogenerator ${}_AP_B$ such that the canonical map, given by the left $A$-action, $\lambda: A \to \shom_B(P,P)$ is an isomorphism.
{\rm{(}}By Corollary \ref{yycor5.6},
$P$ is equivalently an $A$-sprogenerator where the canonical sring morphism given by the right $B$-action is an isomorphism.{\rm{)}}

        \item[(3)] There exists an invertible $(A,B)$-sbimodule $_AP_B$.
    \end{enumerate}
    In particular, 
    any $R\sMod$-enriched equivalence $G: B\sMod \to A\sMod$ is naturally isomorphic to $_AP_B \sotimes_B -.$ 
\end{theorem}

\begin{proof}
    (1) $\implies$ (2) By Theorem \ref{yythm5.4},
    there is a Morita context 
    $$(A, B, {}_AP_B, {}_BQ_A, f, g)$$
    where $f,g$ are surjective.
    It follows from Theorem \ref{yythm5.5} that $P_B$ is a $B$-sprogenerator such that the canonical map $\lambda: A \to \shom_B(P,P)$ is an isomorphism.

    (2) $\implies$ (3) Consider the canonical Morita context associated with $P_B$
    $$( S_P,B,  {}_{S_P}P_B,{}_B(Q_P){}_{S_P}, \mathfrak{a}, \mathfrak{b})$$
    where $\mathfrak{a}, \mathfrak{b}$ are isomorphisms.
    Then $Q_P$ is the inverse of $P$ and $P$ is an invertible $({S_P},B)$-sbimodule.
    Since the canonical map $\lambda: A\to S_P$ is an isomorphism,
    we identify $P$ as an invertible $(A,B)$-sbimodule as we did in the proof of Corollary \ref{yycor5.6}.

    (3) $\implies$ (1) If $_AP_B$ is invertible,
    let $Q$ be the inverse.
    Then it is easy to verify $F=Q\sotimes_A -$ and $G= P \sotimes_B -$ determine an $R\sMod$-enriched equivalence.
\end{proof}

\begin{definition}
\label{yydef5.9}
Let $A$ and $B$ be $R$-salgebras.
We say $A$ and $B$ are 
{\it simplicial Morita equivalent}
if any one of the equivalent 
conditions in Theorem 
\ref{yythm5.8} is satisfied.
\end{definition}

\subsection{A Proposition}
\label{yysec5.3}
The following generalizes \cite[Proposition 4.4.2]{Ford2017}.

\begin{proposition}
\label{yypro5.10}
    Let $S$ be a commutative $R$-salgebra and $A$ be an $S$-salgebra.
    \begin{enumerate}
        \item[(1)] If $A$ is separable over $S$ and $S$ is separable over $R$,
        then $A$ is separable over $R$.

        \item[(2)] If $A$ is separable over $R$,
        then $A$ is separable over $S$.

        \item[(3)] 
        If $A$ is separable over $R$ and $A$ is an $S$-sprogenerator,
        then $S$ is separable over $R.$
    \end{enumerate}
\end{proposition}

\begin{proof}   
    (1) By Proposition \ref{yypro3.16},
    $\mathcal{U}_{S/R}:S\sMod \to R\sMod$ is a separable $R\sMod$-enriched functor and $\mathcal{U}_{A/S}: A\sMod \to S\sMod$ is a separable $S\sMod$-enriched functor.
    By restriction of scalars,
    $\mathcal{U}_{A/S}$ is a separable $R\sMod$-enriched functor.
    The composition of separable enriched functors is separable (for ordinary separable functors, see \cite[Lemma 1.1(2)]{NdBvO1989}).
    The assertion follows from Proposition \ref{yypro3.16} since $\mathcal{U}_{A/R}=\mathcal{U}_{S/R} \circ \mathcal{U}_{A/S}: A\sMod \to R\sMod$
    is a separable $R\sMod$-enriched functor. 

    (2) 
    Let $e\in (A\sotimes_R A^o)_0$ be a separability idempotent.
    Given the canonical sring morphism 
    $\pi: A\sotimes_R A^o \to A\sotimes_S A^o$, 
    define $\overline{e} := \pi(e).$
    Since $\mu_{A/R} = \mu_{A/S} \circ \pi,$
    we get $\mu_{A/S}(\overline{e}) = \mu_{A/S}(\pi(e))=\mu_{A/R}(e)=1_{A_0}.$
    Moreover,
    for every $a\in A_n$,
    \begin{align*}
        (a \otimes_S 1_{A_{n}}) s_0^n(\overline{e}) &= \pi((a \otimes_R 1_{A_{n}})s_0^n(e))\\
        &=\pi((1_{A_n} \otimes_R a)s_0^n(e))= (1_{A_n} \otimes_S a)s_0^n(e).
    \end{align*}
    So
    $\overline{e}$ is a separability idempotent of $A$ as $S$-salgebra.
    Then $A$ is separable over $S.$
   
    (3) 
    Let $A^e=A \sotimes_R A^{o}$ and $S^e = S \sotimes_R S.$
    Since $A$ is an $S$-sprogenerator,
    $A$ is $S$-basic projective and
    $A$ is a direct summand of some $ S^{\oplus m}$ as $S$-smodules by Lemma \ref{yylem5.3}(1).
    Since $A^{o}=A$ as $S$-smodules,
    the tensor product
    $A^e $ is a direct summand of $(S^e)^{\oplus m^2}$ as $S^e$-smodules.
    By Lemma \ref{yylem5.3}(1), 
    $A^e$ is $S^e$-basic projective.
    In particular, $A^e$ is $S^e$ projective by Lemma \ref{yylem2.4}.
    
    Since $A$ is separable over $R$,
    by Proposition \ref{yypro2.7}(2),
    $A$ is a direct summand of $A^e$ as left $A^e$-smodules.
    By restricting the action along $\eta_{A/S}: S \to A$,
    $A$ is a direct summand of $A^e$ as left $S^e$-smodules.
    So $A$ is $S^e$-projective (Lemma \ref{yylem2.4}). 

    Finally, since $A$ is an $S$-sgenerator,
    $S$ is a direct summand of some $A^{\oplus t}$ as $S$-smodules.
    Since $S$ is commutative and the $S$-action on $A$ is central,
    any $S$-smodule structure is equivalently the canonical left $S^e$-smodule structure.
    So $S$ is a direct summand of $A^{\oplus t}$ as $S^e$-smodules,
    which implies $S$ is a projective $S^e$-smodule (as $A$ is a projective $S^e$-smodule).
    So $S$ is separable over $R$ by Proposition \ref{yypro2.7}(1).
\end{proof}

\section{Azumaya Simplicial Algebras}
\label{yysec6}
We define Azumaya simplicial algebras and prove several structure theorems.

\subsection{Basics of Azumaya simplicial algebras}
\label{yysec6.1}

Let $R$ be a commutative sring.

\begin{definition}
\label{yydef6.1}
Let $A$ be an $R$-salgebra.
We say 
\begin{enumerate}
\item[(1)] 
$A$ is {\it central over $R$} 
if $R \cong Z(A)$ under the 
canonical sring morphism 
$\eta_{A/R}: R \to A.$
\item[(2)] $A$ is an 
{\it Azumaya $R$-salgebra}, or 
{\it Azumaya over $R$},
if $A$ is central separable over $R$.
\end{enumerate}
\end{definition}

\begin{remark}
\label{yyrem6.2}
Let $A$ and $B$ be Azumaya $R$-salgebras.
    \begin{enumerate}
        \item[(1)] Recall from Remark \ref{yyrem2.9}(3) that $A_n$ is separable over $R_n$.
        Then by Proposition \ref{yypro4.6}, $A_n$ is Azumaya over $R_n$ since $Z(A_n)=Z(A)_n\cong R_n$.

        \item[(2)] By \eqref{E3.10.1} and Remark \ref{yyrem3.8}(2),
        $R \cong  Z(A)= A^A = eA$ is a direct summand of $A$ as $R$-smodules where $e$ is a separability idempotent.

        \item[(3)] The tensor product $A\sotimes_R B$ is Azumaya over $R.$
        Indeed, by Lemma \ref{yylem4.2},
        $A \sotimes_R B$ is separable over $R\sotimes_R R \cong R$,
        and by Corollary \ref{yycor4.7}
        $$Z(A \sotimes_R B) = Z(A) \sotimes_R Z(B) \cong R\sotimes_R R \cong R.$$
    \end{enumerate}
\end{remark}

\begin{lemma}
\label{yylem6.3}
Let $A$ be an Azumaya $R$-salgebra.
\begin{enumerate}
\item[(1)] 
There is a one-to-one 
correspondence between sideals $I$ of $R$ and sideals $J$ of $A$ given by
        $I \mapsto IA$ and $J\mapsto J\cap R.$

        \item[(2)] If $M$ is a maximal sideal of $A$,
        then there exists a maximal sideal $I$ of $R$ such that $IA=M.$
    \end{enumerate}
    
\end{lemma}

\begin{proof}
    (1) It suffices to show that 
    $$IA \cap R = I \qquad (J\cap R)A = J$$
    which is equivalent to showing for every $n\geq 0$
    $$(IA )_n\cap R_n = I_n \qquad (J\cap R)_nA_n = J_n.$$
    Since $I_n$ is an ideal of $R_n$ and $J_n$ is an ideal of $A_n$, 
    the result then follows from \cite[Corollary 7.1.7]{Ford2017} 
    as $A_n$ is Azumaya over $R_n$ (Remark \ref{yyrem6.2}(1)).

    (2) Let $M$ be a maximal (proper) sideal of $A$ and
    set $I = M \cap R.$
    By part (1),
    $M=(M\cap R)A=IA$.
    If $I\subseteq I'\subseteq R$ for some sideal $I'$ of $R,$
    then 
    $$M= IA\subseteq I'A \subseteq A.$$
    But $M$ is maximal.
    So $I'A = M$ or $A.$
    This means 
    $$I'=I'A\cap R= M\cap R \text{ or } A\cap R = I \text{ or } R$$
    which shows $I$ is maximal.
\end{proof}

\subsection{Auslander-Goldman Theorem}

\begin{theorem}
\label{yythm6.4}
{\rm{(Unifying Theorem)}}
Let $A$ be an $R$-salgebra.
The following are equivalent.
\begin{enumerate}
\item[(1)] 
$A$ is Azumaya (central separable) over $R$.
\item[(2)] 
$A$ is an $A^e$-sprogenerator and $A$ is central over $R$.
\item[(3)] 
$A$ is an $R$-sprogenerator and the canonical sring morphism $A^e \to \shom_R(A,A)$ given by left $A^e$-actions is an isomorphism of $R$-salgebras.
\end{enumerate}
\end{theorem}

\begin{proof}
(1) $\implies$ (2) Suppose $A$ is      Azumaya over $R$.
    Then $A$ is central over $R$.
    It remains to show $A$ is an $A^e$-sprogenerator.
    By Proposition \ref{yypro2.7}(2), $A$ is a direct summand of $A^e$ as left $A^e$-smodule,
    which means $A$ is $A^e$-basic projective (Lemma \ref{yylem5.3}(1)).
    So it remains to show $A$ is an $A^e$-sgenerator,
    which is equivalent to show $A^e$ is a direct summand of $A^{\oplus m}$.
    
    Let $e\in (A^e)_0$ be a separability idempotent. 
    We claim
    \begin{align*}
        A^eeA^e = A^e.
    \end{align*}
    If not, there is a maximal sideal $M$ where $A^eeA^e\subseteq M \subseteq A^e$.
    Note that $A^e$ is Azumaya over $R$ by Remark \ref{yyrem6.2}(3).
    Therefore, by Lemma \ref{yylem6.3}(2),
    $$M = IA^e$$
    for some maximal sideal $I\subseteq R.$
    It follows that
    $$1_{A_0} = \mu_{A/R}(e)\in \mu(IA^e)\subseteq IA.$$
    Therefore, $IA = A$.
    It follows that 
    $$A^e = A \sotimes_R A^{o} =(IA)\sotimes_R A^{o} = I(A \sotimes_R A^{o} )=IA^e= M.$$
    It contradicts that $M$ is a proper sideal.
    Thus $A^eeA^e= A^e$ and $e$ is a full idempotent.
    Then, $A^ee$ is an $A^e$-sprogenerator (for ordinary rings and modules, see \cite[Remark 18.10(D)]{Lam1999}).
    Specifically,
    there exist finitely many $u_i, v_i\in (A^e)_0$ such that 
     \begin{align*}
         1_{(A^e)_0} = \sum_{i=1}^m u_iev_i.
    \end{align*}
     Define left $A^e$-smodule morphisms
     \begin{align*}
         &h: A^e \to (A^ee)^{\oplus m} && x\mapsto (xs_0^n(u_1)s_0^n(e), \dots, xs_0^n(u_m)s_0^n(e))\\
         & f: (A^ee)^{\oplus m}  \to A^e && (y_1,\dots, y_m) \mapsto \sum_{i=1}^m y_is_0^n(v_i)
     \end{align*}
     Then $fh=Id_{A^e}$ which implies 
    $A^e$ is a direct summand of $(A^ee)^{\oplus m}.$
     
     Since $e$ is an idempotent,
     $A^e\cong A^ee\oplus A^e(1_{(A^e)_0}-e)$ as left $A^e$-smodules.
     Recall $\mu_{A/R}: A^e \to A$ is surjective with kernel $J_{A/R}=A^e(1_{(A^e)_0}-e)$ (Proposition \ref{yypro2.7}).
     Therefore,
     $$A \cong A^e/J_{A/R} \cong A^e/A^e(1_{(A^e)_0}-e)\cong A^ee$$ as left $A^e$-smodules.
     It follows $A^e$ is a direct summand of $(A^ee)^{\oplus m}\cong A^{\oplus m}$.
     Hence $A$ is an $A^e$-sprogenerator.

      (2) $\implies$ (1) 
    Suppose $A$ is an $A^e$-sprogenerator and central over $R$.
    Then $A$ is a projective left $A^e$-smodule (Lemma \ref{yylem5.3}(1) and Lemma \ref{yylem2.4})
    which implies $A$ is separable over $R$ by Proposition \ref{yypro2.7}.
    The assertion follows.

    (2) $\iff$ (3) Note that $A$ being central over $R$ is equivalent to $$ R\cong  Z(A) = A^A\cong \shom_{A^e}(A,A).$$
    Applying Theorem \ref{yythm5.8} and Corollary \ref{yycor5.6} to ${}_{A^e}A_R$,
    both conditions (2) and (3) are equivalent to 
    that $R$ being simplicial Morita equivalent to $A^e.$
\end{proof}

\begin{remark}
\label{yyrem6.5}
By Theorem \ref{yythm6.4},
$A$ is Azumaya over $R$ if and only if $A^e=A\sotimes_R A^{o}$ is simplicial Morita equivalent to $R.$
So the central-separable approach agrees with the Morita-theoretical approach when defining Azumaya simplicial algebras.
Therefore,
Theorem \ref{yythm6.4} answers the remaining questions of Pareigis \cite[p.~121]{PareigisIV} and Borceux and Vitale \cite[p.~458]{BV2002} in the setting of simplicial algebras without assuming the extra conditions as Van Oystaeyen and Y. Zhang did in \cite[§4]{VOZ1998}.
\end{remark}

\begin{theorem}
\label{yythm6.6}
    An $R$-salgebra $A$ is separable if and only if $A$ is Azumaya over $Z(A)$ and $Z(A)$ is separable over $R$.
\end{theorem}

\begin{proof}
    ($\implies$) Suppose $A$ is separable over $R$.
    Then $A$ is separable over $Z(A)$ by Proposition \ref{yypro5.10}(2), and therefore $A$ is Azumaya over $Z(A).$
    By Theorem \ref{yythm6.4},
    $A$ is a $Z(A)$-sprogenerator. 
    It follows that $Z(A)$ is separable over $R$ by Proposition \ref{yypro5.10}(3).

    ($\impliedby$) This is Proposition \ref{yypro5.10}(1).
\end{proof}

\subsection{Classification type results}
\label{yysec6.3}

\begin{proposition}
\label{yypro6.7}
Let $A$ be an $R$-salgebra.
Then $A$ is Azumaya over $R$ if and only if $A \cong R \sotimes_{R_0} A_0$ as $R$-salgebras and $A_0$ is Azumaya over $R_0$.
\end{proposition}

\begin{proof}
    ($\implies$) Since $A$ is Azumaya over $R,$
    $A$ is $R$-basic projective by Theorem \ref{yythm6.4}.
    Therefore,
    we have a retraction of $R$-smodules
    \begin{align*}
        A \xrightarrow[]{\jmath} R^{\oplus m} \xrightarrow[]{\pi} A.
    \end{align*}
    Note $R$, $A_0$ (viewed as a constant sring), and $ A$ are $R_0$-salgebras. 
    Let $\eta: R\to A$ and $\iota: A_0 \to A$ be the canonical $R_0$-salgebra morphisms. 
    Define a simplicial map
    \begin{align}
    \label{E6.7.1}
    \tag{E6.7.1}
        f :\quad R \sotimes_{R_0} A_0 \xrightarrow[]{\eta \sotimes_{R_0} \iota} A \sotimes_{R_0} A \xrightarrow[]{\mu_A} A.
    \end{align}
    Then $f$ is an $R$-salgebra morphism since
    \begin{align*}
        f((r_1\otimes_{R_0} a_1)(r_2\otimes_{R_0} a_2)) &= \eta(r_1 r_2)s_0^n(a_1a_2)\\
        &=\eta(r_1)s_0^n(a_1)\eta(r_2)s_0^n(a_2)\\
        &=f(r_1\otimes_{R_0} a_1) f(r_2\otimes_{R_0} a_2)
    \end{align*}
    for every $r_1\otimes a_1, r_2\otimes a_2\in (R\sotimes_{R_0} A_0 )_n$, and clearly
    \begin{align*}
        f(1_{R_n} \otimes_{R_0} 1_{A_0}) = 1_{R_n} s_0^n(1_{A_0}) = 1_{A_n}.
    \end{align*}
    In particular,
    $f$ is an $R$-smodule morphism.
    
    Let $u_1, \dots, u_m$ be standard basis vectors of $R^{\oplus m}_0.$
    Set $a_j= \pi(u_j)\in A_0.$
    Then
    \begin{align*}
        \pi(r_1, \dots, r_m) = \sum_{j=1}^m \eta(r_j)s_0^n(a_j)
    \end{align*}
    for any $(r_1,\dots, r_m)\in (R^{\oplus m})_n.$ Since $\pi \jmath= Id_A,$
    for every $x\in A_n$
    \begin{align}
    \label{E6.7.2}
    \tag{E6.7.2}
        x = \pi (\jmath(x)_1, \dots, \jmath(x)_m) =\sum_{j=1}^m  \eta(\jmath(x)_j) s_0^n(a_j).
    \end{align}
    Define another morphism
    \begin{align}
    \label{E6.7.3}
    \tag{E6.7.3}
        g: A \to R\sotimes_{R_0} A_0 \qquad x\mapsto\sum_{j=1}^m  \jmath(x)_j \otimes_{R_0} a_j
    \end{align}
   for every $x\in A_n$.
   We claim $g$ is an $R$-smodule morphism.
   
   Indeed, for any $\theta:[k] \to [n]$ in $\Delta$ 
    \begin{align*}
        \theta^*(g(x)) &= \sum_{j=1}^m  \theta^*(\jmath(x)_j)\otimes_{R_0} a_j =\sum_{j=1}^m  \jmath(\theta^*(x)\big)_j \otimes_{R_0} a_j =g(\theta^*(x))
    \end{align*}
    which implies $g$ is a simplicial map.
    Furthermore,
    for every $r\in R_n,$
    \begin{align*}
        g(rx) &= \sum_{j=1}^m \jmath(rx)_j \otimes_{R_0} a_j =r(\sum_{j=1}^m  \jmath(x)_j \otimes_{R_0} a_j)= rg(x)
    \end{align*}
    which shows $g$ is an $R$-smodule morphism.

    We show $f$ and $g$ are $R$-smodule isomorphisms.
    Indeed, for every $x\in A_n$
    \begin{align*}
        fg(x) &= f (\sum_{j=1}^m  \jmath(x)_j \otimes_{R_0} a_j) &&\text{by \ref{E6.7.1}}\\
        &=\sum_{j=1}^m \eta(\jmath(x)_j)s_0^n(a_j)=x. &&\text{by \ref{E6.7.2}}
    \end{align*}
    Therefore, $fg= Id_A$.

    Conversely,
    for every $r\otimes_{R_0} a\in (R\sotimes_{R_0} A_0)_n$,
    \begin{align*}
        gf(r\otimes_{R_0} a) &= g(rs_0^n(a)) = rg(s_0^n(a)) && \text{as $g$ is an $R$-smodule morphism}\\
        &=r\sum_{j=1}^m s_0^n(\jmath(a)_j) \otimes_{R_0} a_j && \text{by \ref{E6.7.3}}\\
        &=r \otimes_{R_0} \sum_{j=1}^m  \jmath(a)_j a_j && \text{since $R_n$ is an $R_0$-module}\\
        &=r \otimes_{R_0} \pi \jmath(a) = r\otimes a.
    \end{align*}
    Then $gf=Id_{R \sotimes_{R_0} A_0}$.
    Therefore, $f$ is an $R$-salgebra isomorphism.
    
    Lastly, we note that $A_0$ is Azumaya over $R_0$ by Remark \ref{yyrem6.2}(1).

    ($\impliedby$) Since $A_0$ is separable over $R_0$,
    we note that $A \cong 
    R \sotimes_{R_0} A_0$  is separable over $R$ by Lemma \ref{yylem4.1}.
    Moreover, as $A_0$ is Azumaya over $R_0$,
    then
    $A_n\cong R_n \otimes_{R_0} A_0 $ is also Azumaya over $R_n$ \cite[Corollary 7.1.9]{Ford2017}.
    Therefore, by Proposition \ref{yypro4.6},
    $$Z(A)_n=Z(A_n)\cong R_n.$$
    Then $Z(A)\cong R$ and $A$ is Azumaya over $R.$
    \end{proof}

\begin{corollary}
\label{yycor6.8}
Let $A$ be an 
$R$-salgebra.
The following are equivalent.
\begin{enumerate}
\item[(1)] 
$A$ is separable over $R$.
\item[(2)]  $Z(A)$ is separable over $R$, 
$A_0$ is an Azumaya $Z(A)_0$-algebra, and
$$A\cong Z(A) \sotimes_{Z(A)_0} A_0$$ as $Z(A)$-salgebras.
\end{enumerate} 
\end{corollary}

\begin{proof}
        ($\implies$) If $A$ is separable over $R$,
        then $A$ is Azumaya over $Z(A)$ and $Z(A)$ is separable over $R$ by Theorem \ref{yythm6.6}.
        By Proposition \ref{yypro6.7}
        $A \cong Z(A) \sotimes_{Z(A)_0} A_0$ as $Z(A)$-salgebras.
        Moreover, $A_0$ is Azumaya over $Z(A_0)$ which is $Z(A)_0$ by Proposition \ref{yypro4.6}.

        ($\impliedby$) Since $A_0$ is separable over $Z(A)_0$,
        by Lemma \ref{yylem4.1},
        $A \cong Z(A)\sotimes_{Z(A)_0} A_0$ is separable over $Z(A)$.
        But $Z(A)$ is separable over $R.$
        Therefore, $A$ is separable over $R$ by Proposition \ref{yypro5.10}(1).
    \end{proof}
   
\begin{remark}
\label{yyrem6.9}
By Corollary \ref{yycor6.8}, the data of a separable $R$-salgebra consists of 
\begin{enumerate}
    \item[(1)]  a commutative separable $R$-salgebra (the center $Z(A)$).

    \item[(2)] an ordinary Azumaya $Z(A)_0$-algebra (the algebra $A_0$).
\end{enumerate} 
So to classify separable $R$-salgebras, 
we start with classifying commutative separable $R$-salgebras (those arise as the center) and ordinary Azumaya $Z(A)_0$-algebras.
The latter is controlled by the Brauer group $Br(Z(A)_0)$ up to Morita equivalence.
\end{remark}

Which commutative $R$-salgebras can arise as center of a separable $R$-salgebra?
They are precisely the commutative separable ones.

\begin{lemma}
\label{yylem6.10}
    Let $C$ be a commutative $R$-salgebra.
    Then $C\cong Z(A)$ for some separable $R$-salgebra $A$ if and only if $C$ is separable over $R.$
\end{lemma}

\begin{proof}
    ($\implies$) Suppose $A$ is separable over $R$ and $C\cong Z(A).$
    Then $Z(A)$ is separable over $R$ by Theorem \ref{yythm6.6}.

    ($\impliedby$) If $C$ is commutative and separable over $R$,
    then $C= Z(C).$
\end{proof}

This leads to the problem of classifying commutative separable $R$-salgebras.

\begin{theorem}
\label{yythm6.11}
    Let $C$ be a commutative $R$-salgebra.
    The following are equivalent.
    
    \begin{enumerate}
        \item[(1)] $C$ is separable over $R$.
        
        \item[(2)] $C_0$ is separable over $R_0$ and the canonical sring morphism
        $$\phi: R \sotimes_{R_0} C_0 \to C \qquad  r\otimes c \mapsto rs_0^n(c)$$
        for every $r\otimes c 
        \in (R\sotimes_{R_0} C_0)_n$
        is a ring epimorphism levelwise.
    \end{enumerate}
    In particular, when $C_n$ is a finitely generated $R_n$-module for every $n$,
        $C$ is separable exactly when 
        $C \cong (R \sotimes_{R_0} T)/J$
        where $T$ is a commutative separable $R_0$-algebra, finitely generated as $R_0$-module, 
        and $J$ is an sideal of $R \sotimes_{R_0} T$ with $J_0=0.$ 
\end{theorem}

\begin{proof}
Let $A = R \sotimes_{R_0} C_0$.
Then there is a canonical sring morphism $$R\xrightarrow[]{\eta_{A/R}} A \xrightarrow[]{\phi} C.$$

    (1) $\implies$ (2) If $C$ is separable over $R$,
    then $C_0$ is separable over $R_0.$
    We show \begin{align}
    \label{E6.11.1}
    \tag{E6.11.1}
        \phi_n: A_n=R_n \otimes_{R_0} C_0 \to C_n \qquad r\otimes c \mapsto rs_0^n(c)
    \end{align}
    for every $r\otimes c 
    \in (R\sotimes_{R_0} C_0)_n$
        is a ring epimorphism for every $n\geq 0$.
    Consider the canonical quotient map $$\pi_n: C_n \otimes_{R_n} C_n \to C_n \otimes_{A_n} C_n.$$
    Let $e = \sum_{i=1}^t x_i \otimes_{R_0} y_i \in (C^e)_0$ be a separability idempotent of $C$ over $R.$
    Note that $\phi_n(1_{R_n } \otimes_{R_0} x_i) = s_0^n(x_i).$
    It follows that
    \begin{align*}
        \pi_n(s_0^n(e)) &= \sum_{i=1}^t s_0^n(x_i) \otimes_{A_n} s_0^n(y_i)\\ &=\sum_{i=1}^t \phi_n(1_{R_n } \otimes_{R_0} x_i) \otimes_{A_n} s_0^n(y_i) \\
        &=\sum_{i=1}^t 1_{C_n} \otimes_{A_n} \phi_n(1_{R_n } \otimes_{R_0} x_i)s_0^n(y_i) && \text{$C$ is an $A$-smodule by $\phi: A \xrightarrow[]{} C$}\\
        &=\sum_{i=1}^t 1_{C_n} \otimes_{A_n} s_0^n(x_i)s_0^n(y_i) = 1_{C_n} \otimes_{A_n} 1_{C_n}. && \text{since $\mu(e)=1_{C_0}$}
    \end{align*}
     Since $s_0^n(e)$ is a separability idempotent of $C_n$ over $R_n,$ 
     $$1_{C_n} \otimes_{A_n} 1_{C_n} =\pi_n(s_0^n(e))$$ is a separability idempotent of $C_n$ over $A_n$ (Proposition \ref{yypro5.10}(2)).
     It follows that 
     \begin{align}
     \label{E6.11.2}
     \tag{E6.11.2}
         c\otimes_{A_n} 1_{C_n} = c(1_{C_n} \otimes_{A_n} 1_{C_n} ) = (1_{C_n} \otimes_{A_n} 1_{C_n})c= 1_{C_n}  \otimes_{A_n} c
     \end{align} for every $c\in C_n.$
     Consequently,
     the multiplication map
     $$\mu=\mu_{C_n/A_n}:C_n \otimes_{A_n} C_n \to C_n $$
     is an isomorphism with inverse $$h: C_n \to C_n \otimes_{A_n} C_n \qquad c\mapsto c \otimes_{A_n} 1_{C_n}$$
     for every $ c  \in C_n$.
     Indeed,
     it is clear that $\mu h= Id_{C_n}$ and by \eqref{E6.11.2}
     \begin{align*}
         h\mu(c\otimes_{A_n} c')&=h(cc')=cc'\otimes_{A_n} 1_{C_n} =(c\otimes_{A_n} 1_{C_n}) (c'\otimes_{A_n} 1_{C_n})\\
         &=(c\otimes_{A_n} 1_{C_n}) (1_{C_n}\otimes_{A_n} c')=c\otimes_{A_n } c'.
     \end{align*}
     By \cite[Tag 04VN]{Stacks-04VN},
     $\phi_n: A_n \to C_n$ is an epimorphism.

     (2) $\implies$ (1)
     Since $\phi_n$ is a ring epimorphism,
     by \cite[Tag 04VN]{Stacks-04VN},
     the multiplication map $\mu_{C_n/A_n}:C_n \otimes_{A_n} C_n \to C_n$ is an isomorphism.
     Therefore, we get $\mu_{C/A}: C \sotimes_A C \to C$ is an isomorphism of srings.
     View both as  $C \sotimes_A C$-smodules.
     Then $\mu_{C/A}$ is a $C \sotimes_A C$-smodule isomorphism which implies $C$ is separable over $A$ by Proposition \ref{yypro2.7}(2).
     Since $C_0$ is separable over $R_0$,
     then $A=R \sotimes_{R_0} C_0$ is separable over $R$ by Lemma \ref{yylem4.1}.
     Thus, $C$ is separable over $R$ by Proposition \ref{yypro5.10}(1).

     Suppose, in addition, that every $C_n$ is a finitely generated $R_n$-module.
     Then 
     $C_n$ is a finitely generated $A_n$-module.
     Therefore,
     the ring morphism $\phi_n: A_n  \to C_n$ is an epimorphism if and only if $\phi_n$ is surjective \cite[Tag 04VT]{Stacks-04VT}.
     Then it follows that $\phi: A \to C$ is a surjective sring morphism.
     Let $J=\ker \phi$.
     Then
     $C\cong A/J.$
     In level $0,$
     the canonical map
     $\phi_0: R_0 \otimes_{R_0} C_0 \to C_0$ is clearly injective.
     So $J_0=0.$
     Conversely,
     if $C \cong (R\sotimes_{R_0} T)/J$,
     then $R\sotimes_{R_0} T$ is separable over $R$ by Lemma \ref{yylem4.1}.
     It follows that $C \cong (R\sotimes_{R_0} T)/J$ is separable over $R$ by Lemma \ref{yylem4.4}.
     \end{proof}

\begin{remark}
\label{yyrem6.12}
We record two consequences of 
Theorem \ref{yythm6.11}. 
         \begin{enumerate}
             \item[(1)] 
             Let $K$ be a field. Recall in Example \ref{yyexa4.9},
             $W^{\langle 0 \rangle}/K$ is not separable over $K$.
             Let $C=W^{\langle 0 \rangle}/K$ and $R=K\Delta^0$.
             The canonical ring morphism
             $$\phi: R_n \sotimes_{R_0} (W^{\langle 0 \rangle}/K)_0 \to  (W^{\langle 0 \rangle}/K)_n $$
             is $K \to K^{\oplus (n+1)}$ which is not an epimorphism when $n\geq 1$.
             So Theorem \ref{yythm6.11} reconfirms the non-separability of $W^{\langle 0 \rangle}/K$ over $K.$

             \item[(2)] Let $\Bbbk$ be a commutative ring.
             Let $R=\Bbbk\Delta^0$.
             Suppose $C$ is a commutative separable $R$-salgebra such that each $C_n$ is a finitely generated $R_n$-module.
             Then Theorem \ref{yythm6.11} implies $C \cong \Bbbk\Delta^0 \sotimes_{\Bbbk} C_0$ because any sideal $J$ of $\Bbbk\Delta^0\sotimes_{\Bbbk} C_0$ such that $J_0=0$ must be trivial.
             So Theorem \ref{yythm6.11} recovers Theorem \ref{yythm2.11} for the finite commutative $\Bbbk$-salgebras.
         \end{enumerate}
         
     \end{remark}

\section{Brauer Groups of Simplicial Commutative Rings}
\label{yysec7}

We define the Brauer group of a commutative sring $R$ and show it is isomorphic to the Brauer group of the level zero component $R_0.$

\subsection{Basics of Brauer groups}
\label{yysec7.1}

Let $R$ be a commutative sring.
Define $B(R)$ to be the set of isomorphism classes of Azumaya $R$-salgebras.
Then $(B(R), \sotimes_R)$ is a commutative monoid under $\sotimes_R$ with unit $R$.
For any $A,B\in B(R)$ (by abuse of notation,
we use $A$ to denote the isomorphism class of $A$),
define $A \sim B$ if $A$ and $B$ are simplicial Morita equivalent.

\begin{lemma}
\label{yylem7.1}
The reltion 
$\sim$ respects $\sotimes_R$ and is an equivalence relation in $B(R).$
\end{lemma}

\begin{proof}
    We show if $A\sim A'$ and $B \sim B'$, 
    then $(A\sotimes_R B) \sim (A' \sotimes_R B').$
    By Theorem \ref{yythm5.8},
    it suffices to show there exists an invertible $(A\sotimes_R B, A' \sotimes_R B')$-sbimodule.
    Since $A\sim A'$ and $B \sim B'$, by Theorem \ref{yythm5.8},
    there is an invertible $(A,A')$-sbimodule $P$ and an invertible $(B,B')$-sbimodule $P'$.
    It follows that $P\sotimes_R P'$ is the desired invertible $(A\sotimes_R B, A' \sotimes_R B')$-sbimodule.
    Thus, $\sim$ respects levelwise tensor product $\sotimes_R$.
    The usual reflexivity, symmetry, and transitivity are routine to verify.
\end{proof}

\begin{definition}
\label{yydef7.2}
    Define $Br(R):=B(R)/\sim$.
    This is a commutative group with unit $R$,
    product $\sotimes_R$,
    and inverse $(-)^{o}$ (Remark \ref{yyrem6.5}).
    We call $Br(R)$ the {\it Brauer group} of $R$ and
    we denote the equivalence class of $A\in B(R)$ by $[A].$
\end{definition}

The following proposition recovers the usual descriptions of Brauer groups; see \cite[p.~108]{VOZ1998} and 
\cite[p.~252]{Ford2017}.

\begin{proposition}
\label{yypro7.3}
    Let $A$ and $B$ be Azumaya $R$-salgebras.
    The following are equivalent.
    \begin{enumerate}
        \item[(1)] $A$ is simplicial Morita equivalent to $B$.

        \item[(2)] There exist $R$-sprogenerators $N$ and $M$ such that 
        $$A\sotimes_R \shom_R(M,M) \cong B \sotimes_R \shom_R(N,N).$$
    \end{enumerate}
    
\end{proposition}

\begin{proof}
    (1) $\implies$ (2) Suppose $A \sim B$.
    By Lemma \ref{yylem7.1} and Theorem \ref{yythm6.4},
    $$(A\sotimes_R B^{o}) \sim (B \sotimes_R B^{o}) =B^e \cong \shom_R(B,B) \sim  R.$$
    By Theorem \ref{yythm5.8},
    there is an $R$-sprogenerator $N$ such that $A\otimes_R B^{o} \cong \shom_R(N,N).$
    Therefore,
    we have
    \begin{align*}
        A \sotimes_R \shom_R(B,B) &\cong A \sotimes_R (B \sotimes_R B^{o}) \cong (A \sotimes_R B^{o} )\sotimes_R B  \\
        & \cong \shom_R(N,N) \sotimes_R B  \cong B \sotimes_R \shom_R(N,N).
    \end{align*}
    The assertion follows from Theorem \ref{yythm6.4} that
    $M:=B$ is an $R$-sprogenerator.

    (2) $\implies $ (1) Suppose $A \sotimes_R \shom_R(M,M) \cong B \sotimes_R \shom_R(N,N)$ for some $R$-sprogenerators $M$ and $N$.
    Then $\shom_R(M,M) \sim R \sim \shom_R(N,N)$ by Theorem \ref{yythm5.4}.
    Therefore,
    $$(A \sotimes_R \shom_R(M,M)) \sim( A \sotimes_R R) \cong A$$
    since $\sim$ respects $\sotimes_R$ by Lemma \ref{yylem7.1}.
    Similarly,
    $$(B \sotimes_R \shom_R(N,N)) \sim (B \sotimes_R R) \cong B.$$
    It follows that
    $A \sim A \sotimes_R \shom_R(M,M) \cong B \sotimes_R \shom_R(N,N)  \sim B.$
\end{proof}

\subsection{Reduction to degree 0}
\label{yysec7.2}

\begin{theorem}
\label{yythm7.4}
{\rm{(Triviality Theorem)}}
Let $R$ be a commutative sring.
Then $$Br(R)\cong Br(R_0).$$
\end{theorem}

\begin{proof}
    If $A$ is Azumaya over $R$,
    then $A_0$ is Azumaya over $R_0$ (Remark \ref{yyrem6.2}(1)).
    Moreover,
    if $A\sim B$,
    then $A_0\sim B_0$ because any invertible $(A,B)$-sbimodule $P$ induces an invertible $(A_0, B_0)$-bimodule $P_0$.
    Thus,
    we define a group homomorphism
    \begin{align*}
        f: Br(R) \to Br(R_0) \qquad [A] \mapsto [A_0].
    \end{align*}
    Conversely,
    given any Azumaya $R_0$-algebra $T$,
    then 
    $R \sotimes_{R_0} T$ is Azumaya over $R$ (Proposition \ref{yypro6.7}).
    Moreover, if $T, T'$ are Morita equivalent $R_0$-algebras,
    there is an invertible $(T,T')$-bimodule $Q$. 
    This induces an invertible $(R\sotimes_{R_0}T, R\sotimes_{R_0}T')$-sbimodule $R\sotimes_{R_0} Q.$
    Therefore, $R\sotimes_{R_0}T \sim R\sotimes_{R_0}T'$.
    Furthermore,
    $$(R \sotimes_{R_0} T) \sotimes
    _R (R\sotimes_{R_0} T') \cong (R \sotimes_R R) \sotimes_{R_0} (T \sotimes_{R_0} T')\cong R\sotimes_{R_0} (T \sotimes_{R_0} T').$$
    This implies that we can define a group homomorphism 
    \begin{align*}
        g:Br(R_0) \to Br(R) \qquad [T] \mapsto [R\sotimes_{R_0} T].
    \end{align*}
    Now $f$ and $g$ are group isomorphisms since $fg([T]) = f([R\sotimes_{R_0} T]) = [T]$ and 
    \begin{align*}
     gf([A]) = g([A_0]) = [R\sotimes_{R_0} A_0]=[A]
    \end{align*}
    where the third equality is by Proposition \ref{yypro6.7}.
    This completes the proof.
\end{proof}

\subsection*{Acknowledgments}
We used GPT to assist in verifying some of our arguments and searching typos.
We have written and checked all proofs on our own, and we take full responsibility for the mathematical content of this paper.
This project was partially supported by the US National Science Foundation (No. DMS-2302087).

\end{document}